\documentclass{amsart}
\usepackage[english]{babel} % Language setting

\usepackage[letterpaper,top=2cm,bottom=2cm,left=3cm,right=3cm,marginparwidth=1.75cm]{geometry}
\usepackage{multicol}
\usepackage[]{caption}
\usepackage[]{subcaption} 
\usepackage{enumitem}
\usepackage{xstring} % for string-substitutions
\usepackage{amssymb}
\usepackage{mathtools}
\usepackage[mathscr]{euscript}
\usepackage{bbm} % (extra blackboard bold letters) try \mathbbm{Aa123}

\usepackage{ifthen}  % needed for `MaxmeanderLower`

\usepackage{graphicx}
\usepackage{tikz}
\tikzset{x=.75cm,y=.75cm}
\usetikzlibrary{cd, arrows, calc, shapes.geometric}

\usepackage[%
    colorlinks,
    linkcolor=blue!80, 
    citecolor=green!65!black, 
    urlcolor=magenta!85!black,
    ]{hyperref}

\definecolor{myred}{rgb}{.8,.1,.1}
\definecolor{myblue}{rgb}{.1,.1,.8}
\definecolor{mygreen}{rgb}{.1,.8,.1}
\definecolor{mygold}{rgb}{0.95,0.8,0}

\definecolor{eigpink}{RGB}{220,30,150}
\definecolor{eigcyan}{RGB}{0,185,215}
\definecolor{eiglav}{RGB}{155,125,215}

\newcommand{\bgr}[1]{\textcolor{mygreen}{\ensuremath{\mathbf{#1}}}}
\newcommand{\gr}[1]{\textcolor{mygreen}{\ensuremath{#1}}}
\newcommand{\bl}[1]{\textcolor{myblue}{\ensuremath{#1}}}
\newcommand{\rd}[1]{\textcolor{red}{\ensuremath{#1}}}

\newcommand*{\getylength}[1]{
  \path let \p{y}=(0,1), \n{ylen}={veclen(\x{y},\y{y})}
    in \pgfextra{\xdef#1{\n{ylen}}};
}

\newcommand{\meanderD}[3][.6]{%
\begin{tikzpicture}[scale=#1]
	\getylength{\mylen}
\tikzset{%
medge/.style = {->, line width=.5pt, black}, %mygreen!50!black,
vertex/.style = {line width=.5pt, font=\footnotesize, text=white, minimum size=0.2*\mylen, inner sep=0.05*\mylen}, %, mygreen!50!black
vertex text/.style = {font=\footnotesize, text=blue, minimum size=0.2*\mylen, inner sep=0.1*\mylen} %, mygreen!50!black
}
\foreach \i in {1, ..., #2} {
	\node[vertex] (\i) at (\mylen*\i,0){8}; % invisible character of width "8"
	\node[vertex text] (n\i) at (\mylen*\i,0){$\mathsf\i$}; 
	}

\foreach \i/\j in {#3} {
    \pgfmathsetmacro{\x}{abs(\j-\i)}
    \draw[medge] (\i) to[bend left=40+5*\x] (\j); 
}
\end{tikzpicture}
}

\newcommand{\meanderDUnderOverlay}[5][.6]{%
\begin{tikzpicture}[scale=#1]
	\getylength{\mylen}
\tikzset{%
medge/.style = {->, line width=.5pt, black}, %mygreen!50!black,
vertex/.style = {line width=.5pt, font=\footnotesize, text=white, minimum size=0.2*\mylen, inner sep=0.05*\mylen}, %, mygreen!50!black
vertex text/.style = {font=\footnotesize, text=blue, minimum size=0.2*\mylen, inner sep=0.1*\mylen} %, mygreen!50!black
}
{#4}
\foreach \i in {1, ..., #2} {
	\node[vertex] (\i) at (\mylen*\i,0){8}; % invisible character of width "8"
	\node[vertex text] (n\i) at (\mylen*\i,0){$\mathsf\i$}; 
	}

\foreach \i/\j in {#3} {
    \pgfmathsetmacro{\x}{abs(\j-\i)}
    \draw[medge] (\i) to[bend left=40+5*\x] (\j); 
}
{#5}
\end{tikzpicture}
}

\newcommand{\maxmeander}[2][.6]{%
\begin{tikzpicture}[scale=#1]
	\getylength{\mylen}
\tikzset{%
medge/.style = {-, line width=.5pt, black}, %mygreen!50!black,
vertex/.style = {line width=.5pt, font=\footnotesize, text=white, minimum size=0.2*\mylen, inner sep=0.05*\mylen},
vertex text/.style = {font=\footnotesize, text=blue, minimum size=0.2*\mylen, inner sep=0.1*\mylen} %, mygreen!50!black
}
\foreach \p/\q in {#2} {  % there should only be one!
    \pgfmathsetmacro{\n}{\p + \q}
    \pgfmathtruncatemacro{\nn}{floor(\n/2)}    
    \pgfmathtruncatemacro{\pp}{floor(\p/2)}    
    \pgfmathtruncatemacro{\qq}{floor(\q/2)}    
    \foreach \i in {1, ..., \n} {
       \node[vertex] (\i) at (\mylen*\i,0){8}; % invisible character of width "8"
	   \node[vertex text] (n\i) at (\mylen*\i,0){$\mathsf\i$}; 
	}

    \foreach \i [evaluate=\i as \j using int(\n+1-\i)] in {1, ..., \nn} {
        \draw[medge] (\i) to[bend left=75] (\j);
    }

    \foreach \i [evaluate=\i as \j using int(\p+1-\i)] in {1, ..., \pp} {
        \draw[medge] (\i) to[bend right=75] (\j); 
    }

    \foreach \i [evaluate=\i as \ii using int(\p+\i), 
                 evaluate=\i as \jj using int(\n+1-\i)] in {1, ..., \qq} {
        \draw[medge] (\ii) to[bend right=75] (\jj); 
    }
    }
\end{tikzpicture}
}

\newcommand{\maxmeanderLower}[2][.6]{%
\begin{tikzpicture}[scale=#1]
	\getylength{\mylen}
\tikzset{%
medge/.style = {-, line width=.5pt, black}, %mygreen!50!black,
vertex/.style = {line width=.5pt, font=\footnotesize, text=white, minimum size=0.2*\mylen, inner sep=0.05*\mylen},
vertex text/.style = {font=\footnotesize, text=blue, minimum size=0.2*\mylen, inner sep=0.1*\mylen} %, mygreen!50!black
}
\foreach \p/\q in {#2} {  % there should only be one!
    \pgfmathsetmacro{\n}{\p + \q}
    \pgfmathtruncatemacro{\nn}{floor(\n/2)}    
    \pgfmathtruncatemacro{\pp}{floor(\p/2)}    
    \pgfmathtruncatemacro{\qq}{floor(\q/2)}    
    \foreach \i in {1, ..., \n} {
       \node[vertex] (\i) at (\mylen*\i,0){8}; % invisible character of width "8"
	   \node[vertex text] (n\i) at (\mylen*\i,0){$\mathsf\i$}; 
	}

    \foreach \i [evaluate=\i as \j using int(\n+1-\i)] in {1, ..., \nn} {
        \draw[medge] (\i) to[bend right=75] (\j);
    }

    \foreach \i [evaluate=\i as \j using int(\p+1-\i)] in {1, ..., \pp} {
        \draw[medge] (\i) to[bend left=75] (\j); 
    }

    \foreach \i [evaluate=\i as \ii using int(\p+\i), 
                 evaluate=\i as \jj using int(\n+1-\i)] in {1, ..., \qq} {
        \draw[medge] (\ii) to[bend left=75] (\jj); 
    }
    }
    \end{tikzpicture}
}

\newcommand{\MaxmeanderLower}[3][.6]{%
\begin{tikzpicture}[scale=#1]
	\getylength{\mylen}
\tikzset{%
medge/.style = {-, line width=.5pt, black}, %mygreen!50!black,
vertex/.style = {line width=.5pt, font=\footnotesize, text=white, minimum size=0.1*\mylen, inner sep=0.02*\mylen},
vertex text/.style = {font=\footnotesize, text=blue, minimum size=0.2*\mylen, inner sep=0.1*\mylen} %, mygreen!50!black
}
\foreach \p/\q in {#2} {  % there should only be one!
    \pgfmathtruncatemacro{\n}{\p + \q}
    \pgfmathtruncatemacro{\px}{\p + 1}    
    \pgfmathtruncatemacro{\qx}{\q + 1}    
    \pgfmathtruncatemacro{\nn}{floor(\n/2)}    
    \pgfmathtruncatemacro{\pp}{floor(\p/2)}    
    \pgfmathtruncatemacro{\qq}{floor(\q/2)}    
    \ifthenelse{#3=0}{%
        \foreach \i in {1, ..., \p} {
           \node[vertex] (\i) at (\mylen*\i,0){o}; % invisible character of width "o"
    	   \node[vertex text] (n\i) at (\mylen*\i,0){\!${}_{}\bs\cdot$}; 
    	}
        \foreach \i in {1, ..., \q} {
           \pgfmathtruncatemacro\ix{\p + \i}
           \node[vertex] (\ix) at (\mylen*\ix,0){o}; % invisible character of width "o"
    	   \node[vertex text] (n\ix) at (\mylen*\ix,0){\!$\mathsf\i$}; 
    	}
    }{%
        \foreach \i in {1, ..., \p} {
           \node[vertex] (\i) at (\mylen*\i,0){o}; % invisible character of width "o"
    	   \node[vertex text] (n\i) at (\mylen*\i,0){$\mathsf\i$\!}; 
    	}
        \foreach \i in {\px, ..., \n} {
           \node[vertex] (\i) at (\mylen*\i,0){o}; % invisible character of width "o"
    	   \node[vertex text] (n\i) at (\mylen*\i,0){$\bs\cdot{}_{}$\!}; 
    	}
    }
    \foreach \i [evaluate=\i as \j using int(\n+1-\i)] in {1, ..., \nn} {
        \draw[medge] (\i) to[bend right=75] (\j);
    }

    \foreach \i [evaluate=\i as \j using int(\p+1-\i)] in {1, ..., \pp} {
        \draw[medge] (\i) to[bend left=75] (\j); 
    }

    \foreach \i [evaluate=\i as \ii using int(\p+\i), 
                 evaluate=\i as \jj using int(\n+1-\i)] in {1, ..., \qq} {
        \draw[medge] (\ii) to[bend left=75] (\jj); 
    }
    }
    \end{tikzpicture}
}

\newcommand{\maxmeanderRetracting}[5][.6]{%
\begin{tikzpicture}[scale=#1]
  \getylength{\mylen}
  \tikzset{%
    medge/.style  = {-, line width=.5pt, black},
    vertex/.style = {line width=.5pt, font=#5, text=white,
                     minimum size=0.2*\mylen, inner sep=0.05*\mylen},
    vertex text/.style = {font=#5, text=blue,
                     minimum size=0.1*\mylen, inner sep=0.1*\mylen}
  }
  \foreach \p/\q in {#2} {
    \pgfmathsetmacro{\n}{\p + \q}
    \pgfmathtruncatemacro{\nn}{floor(\n/2)}
    \pgfmathtruncatemacro{\pp}{floor(\p/2)}
    \pgfmathtruncatemacro{\qq}{floor(\q/2)}

    \pgfmathsetmacro{\mylennum}{\mylen/1pt}

    \foreach \i in {1, ..., \p} {
      \node[vertex]      (\i)  at (\mylennum*\i pt, 0) {8};
      \node[vertex text] (n\i) at (\mylennum*\i pt, 0) {$\mathsf{\i}$};
    }
    
    \foreach \j in {1, ..., \q} {
      \pgfmathtruncatemacro{\pj}{\p + \j}
      \node[vertex text] (f\j) at (\mylennum*\pj pt, 0) {\color{blue!40}{$\mathsf{\pj}$}};
    }

    \pgfmathtruncatemacro{\apex}{ceil(\n/2)}
    \foreach \k in {1, ..., \q} {
      \pgfmathtruncatemacro{\idx}{\p + \k}
      \pgfmathsetmacro{\dist}{\idx - \apex}
      \pgfmathsetmacro{\tx}{\mylennum*\apex + (\dist-#4)*\mylennum*cos(#3)}
      \pgfmathsetmacro{\ty}{(\dist-#4)*\mylennum*sin(#3)}
      \node[vertex]      (\idx) at (\tx pt, \ty pt) {8};
      \node[vertex text] (n\idx) at (\tx pt, \ty pt) {\large$\boldsymbol{\cdot}$};
    }

    \foreach \i [evaluate=\i as \j using int(\n+1-\i)] in {1, ..., \nn} {
      \draw[medge] (\i) to[bend left=40+2*(\j-\i)] (\j);
    }

    \foreach \i [evaluate=\i as \j using int(\p+1-\i)] in {1, ..., \pp} {
      \draw[medge] (\i) to[bend right=50+3*(\j-\i)] (\j);
    }

    \foreach \k in {1, ..., \qq} {
      \pgfmathtruncatemacro{\ia}{\p + \k}
      \pgfmathtruncatemacro{\ib}{\p + \q + 1 - \k}
      \pgfmathtruncatemacro{\bendamt}{50+3*(\ib-\ia)}
      \draw[medge] (\ia) to[bend right=\bendamt] (\ib);
    }
  }
\end{tikzpicture}
}
\newcommand{\gammaS}[3][.6]{%
\raisebox{-.5\height}{%
\begin{tikzpicture}[scale=#1]
	\getylength{\mylen}
\tikzset{%
medge/.style = {->, >=latex', line width=.5pt, black}, %mygreen!50!black,
vertex/.style = {shape=circle, draw, line width=.5pt, font=\footnotesize, text=white, minimum size=0.2*\mylen, inner sep=0.1*\mylen}, %, mygreen!50!black
vertex text/.style = {font=\footnotesize, text=black, minimum size=0.2*\mylen, inner sep=0.1*\mylen} %, mygreen!50!black
}
\foreach \i/\j/\x in {#2} {
	\node[vertex] (\x) at (\i,\j){8}; % invisible character of width "8"
	\node[vertex text] (n\x) at (\i,\j){$\mathsf\x$}; % use two nodes to better control vertex size
	}
\foreach \x/\y in {#3} {
	\draw[medge] (\x) to (\y); % or you could bend it with `to[bend left]`
	}
\end{tikzpicture}
}}
\newcommand{\gammaSUnderOverlay}[5][.6]{%
\raisebox{-.5\height}{%
\begin{tikzpicture}[scale=#1]
	\getylength{\mylen}
\tikzset{%
medge/.style = {->, > = latex', line width=.5pt, black}, %mygreen!50!black,
vertex/.style = {shape=circle, draw, fill=white,line width=.5pt, font=\footnotesize, text=white, minimum size=0.2*\mylen, inner sep=0.1*\mylen}, %, mygreen!50!black
vertex text/.style = {font=\footnotesize, text=black, minimum size=0.2*\mylen, inner sep=0.1*\mylen} %, mygreen!50!black
}
#4
\foreach \i/\j/\x in {#2} {
	\node[vertex] (\x) at (\i,\j){8}; % invisible character of width "8"
	\node[vertex text] (n\x) at (\i,\j){$\mathsf\x$}; % use two nodes to better control vertex size
	}
\foreach \x/\y in {#3} {
	\draw[medge] (\x) to (\y); % or you could bend it with `to[bend left]`
	}
#5
\end{tikzpicture}
}}
\newcommand{\coloredgammaS}[3][.6]{%
\raisebox{-.5\height}{%
\begin{tikzpicture}[scale=#1]
	\getylength{\mylen}
\tikzset{%
medge/.style = {->, > = latex', line width=.5pt, black}, %mygreen!50!black,
vertex/.style = {shape=circle, draw, line width=.5pt, font=\footnotesize, minimum size=0.2*\mylen, inner sep=0.1*\mylen}, %, mygreen!50!black
vertex text/.style = {font=\footnotesize, text=black, minimum size=0.2*\mylen, inner sep=0.1*\mylen} %, mygreen!50!black
}
\foreach \i/\j/\x\c in {#2} {
	\node[vertex,fill=\c, text=\c] (\x) at (\i,\j){8}; % invisible character of width "8"
	\node[vertex text] (n\x) at (\i,\j){$\mathsf\x$}; % use two nodes to better control vertex size
	}
\foreach \x/\y in {#3} {
	\draw[medge] (\x) to (\y); % or you could bend it with `to[bend left]`
	}
\end{tikzpicture}
}}
\newcommand{\orderideals}[4][.6]{%
\raisebox{-.5\height}{%
\begin{tikzpicture}[scale=#1]
\tikzset{%
hedge/.style = {-, line width=.5pt, black}, %mygreen!50!black,
vertex/.style = {black, font=\small, fill=white, inner sep=2pt}, %draw=red, line width=.1pt, 
}
\foreach \i/\j/\x/\l in {#3} {
    \StrSubstitute{\l}{,}{{,}}[\ltight]%
    \if#20
        \node[vertex] (\x) at (\i,\j){$\{\mathsf{\l}\}$};
    \else
        \node[vertex] (\x) at (\i,\j){$\mathsf{\ltight}$};
    \fi
    }
\foreach \x/\y in {#4} {
	\draw[hedge] (\x) to (\y); 
	}
\end{tikzpicture}
}}
\newcommand{\coloredorderideals}[4][.6]{%
\raisebox{-.5\height}{%
\begin{tikzpicture}[scale=#1]
\tikzset{%
hedge/.style = {-, line width=.5pt, black}, %mygreen!50!black,
vertex-bg/.style = {white, shape=circle, font=\small, fill=white, inner sep=1pt}, vertex/.style = {black, font=\small, inner sep=1pt}, %draw=red, line width=.1pt, 
}
\foreach \i/\j/\x/\l/\c in {#3} {
   \StrSubstitute{\l}{,}{{,}}[\ltight]%
    \if#20
        \node[vertex-bg, shape=circle,fill=\c] (bg\x) at (\i,\j){\phantom{8}};
        \node[vertex] (\x) at (\i,\j){$\{\mathsf{\l}\}$};
    \else
        \node[vertex-bg, shape=circle,fill=\c] (bg\x) at (\i,\j){\phantom{8}};
        \node[vertex] (\x) at (\i,\j){$\mathsf{\ltight}$};
    \fi
    }
\foreach \x/\y in {#4} {
	\draw[hedge] (\x) to (\y); 
	}
\end{tikzpicture}
}}
\newcommand{\eigenShapeOverlay}[4][.7]{%
\raisebox{-.5\height}{%
\begin{tikzpicture}[scale=#1]
	\getylength{\mylen}
\tikzset{%
blank/.style = {shape=rectangle, font=\tiny, text=gray!80, minimum size=0.5*\mylen, inner sep=0.1*\mylen},
box/.style = {shape=rectangle, line width=.5pt, white, fill=gray!20, draw},
box text/.style = {font=\footnotesize, text=black, minimum size=0.2*\mylen, inner sep=0*\mylen}
}
\foreach \i/\j/\x in {#2} {
	\draw[box] (\j-.5,-\i-.5) rectangle ++(1,1);
	\node[box text] (\i\j) at (\j,-\i){$\x$};
	}
\foreach \i/\j in {#3}
	\node[blank] (\i\j) at (\j,-\i){$\cdot$};
#4
\end{tikzpicture}
}}
\newcommand{\Shape}[4][.7]{%
\begin{tikzpicture}[scale=#1]
	\getylength{\mylen}
\tikzset{%
blank/.style = {shape=rectangle, font=\tiny, text=gray!80, minimum size=0.5*\mylen, inner sep=0.1*\mylen},
box/.style = {shape=rectangle, line width=.5pt, white, fill=gray!20, draw},
box text/.style = {font=\footnotesize, text=black, minimum size=0.2*\mylen, inner sep=0*\mylen}
}
\foreach \i/\j\x in {#2} {
	\draw[box] (\j-.5,-\i-.5) rectangle ++(1,1);
	\node[box text] (\i\j) at (\j,-\i){\gr{\i{\scriptscriptstyle\,}\j}};
	}
\foreach \i/\j in {#3} {
	\draw[box] (\j-.5,-\i-.5) rectangle ++(1,1);
	\ifthenelse{\i=\j}{\node[box text] (\i\j) at (\j,-\i){\bl{\times}};}{\node[box text] (\i\j) at (\j,-\i){$\times$};}
	}
\foreach \i/\j in {#4}
	\node[blank] (\i\j) at (\j,-\i){$\cdot$};
\end{tikzpicture}
}
\newcommand{\ShapeOverlay}[5][.7]{%
\begin{tikzpicture}[scale=#1]
	\getylength{\mylen}
\tikzset{%
blank/.style = {shape=rectangle, font=\tiny, text=gray!80, minimum size=0.5*\mylen, inner sep=0.1*\mylen},
box/.style = {shape=rectangle, line width=.5pt, white, fill=gray!20, draw},
box text/.style = {font=\footnotesize, text=black, minimum size=0.2*\mylen, inner sep=0*\mylen}
}
\foreach \i/\j\x in {#2} {
	\draw[box] (\j-.5,-\i-.5) rectangle ++(1,1);
	\node[box text] (\i\j) at (\j,-\i){\gr{\i{\scriptscriptstyle\,}\j}};
	}
\foreach \i/\j in {#3} {
	\draw[box] (\j-.5,-\i-.5) rectangle ++(1,1);
	\ifthenelse{\i=\j}{\node[box text] (\i\j) at (\j,-\i){\bl{\times}};}{\node[box text] (\i\j) at (\j,-\i){$\times$};}
	}
\foreach \i/\j in {#4}
	\node[blank] (\i\j) at (\j,-\i){$\cdot$};
#5
\end{tikzpicture}
}
\newcommand{\fenceposet}[2][.6]{%
\raisebox{.5\height}{%
\begin{tikzpicture}[scale=#1]
\tikzset{%
fedge/.style = {-, line width=.5pt, black}, %mygreen!50!black,
vertex/.style = {blue, font=\small, fill=white, inner sep=1pt}, %draw=red, line width=.1pt, 
}
    \draw[fedge] (0,0)
    \foreach \I in {#2} {
      \ifdim\I pt > 0pt -- ++(1,1) \else -- ++(1,-1) \fi
    };
    \node[vertex] at (0,0) {$\mathsf1$};
    \coordinate (current) at (0,0);
    \foreach[count=\t from 2] \I in {#2} {
      \ifdim\I pt > 0pt
        \coordinate (current) at ($(current) + (1,1)$);
      \else
        \coordinate (current) at ($(current) + (1,-1)$);
      \fi
      \node[vertex] at (current) {$\mathsf\t$};
    }
  \end{tikzpicture}%
}}
\newcommand{\fenceposetExtended}[3][.6]{%
\raisebox{.5\height}{%
\begin{tikzpicture}[scale=#1]
\tikzset{%
  fedge/.style   = {-, line width=.5pt, black},
  vertex/.style  = {blue, font=\tiny, fill=white, inner sep=1pt},
  fedgeX/.style  = {-, line width=.6pt, magenta!90!black},
  vertexX/.style = {magenta!90!black, font=\tiny, fill=white, inner sep=.5pt},
}
  \coordinate (F1) at (0,0);
  \foreach \I [count=\k from 2] in {#2} {
    \pgfmathtruncatemacro{\kprev}{\k-1}
    \ifdim\I pt > 0pt
      \coordinate (F\k) at ($(F\kprev)+(1,1)$);
    \else
      \coordinate (F\k) at ($(F\kprev)+(1,-1)$);
    \fi
    \global\edef\maincount{\k}%
  }
  \pgfmathtruncatemacro{\extk}{\maincount+1}
  \ifdim#3pt > 0pt
    \coordinate (F\extk) at ($(F\maincount)+(1,1)$);
  \else
    \coordinate (F\extk) at ($(F\maincount)+(1,-1)$);
  \fi
  \draw[fedge] (F1)
    \foreach \k in {2,...,\maincount} { -- (F\k) };
  \draw[fedgeX] (F\maincount) -- (F\extk);
  \foreach \k in {1,...,\maincount} {
    \node[vertex] at (F\k) {$\mathsf{\k}$};
  }
  \node[vertexX,shape=circle] at (F\extk) {$\mathsf{\extk}$};
\end{tikzpicture}%
}}% try: \fenceposetExtended{1,-1,-1,1,1,1}{-1}

\def\skylineboxsize{0.475}   % cm ? side length of each unit box
\def\skylinegap{0.125}       % cm ? gap between adjacent boxes

\newcommand{\fskyline}[1]{%
  \begin{tikzpicture}[
      x=1cm, y=1cm,
      box/.style={
        draw=black, line width=0.6pt,
        fill=white,
        minimum size=\skylineboxsize cm,
        inner sep=0pt,
      },
      baseline=(current bounding box.south),
  ]
    \foreach \Labels [count=\xi from 0] in {#1} {%
      \foreach \lbl [count=\yi from 1] in \Labels {%
        \pgfmathsetmacro{\bx}{\xi*(\skylineboxsize+\skylinegap)}%
        \pgfmathsetmacro{\by}{(\yi-1)*(\skylineboxsize+\skylinegap)}%
        \node[box] (\lbl) at (\bx,\by) {\small\lbl};%
      }%
    }%
  \end{tikzpicture}%
}

\newcommand{\fcskyline}[1]{%
  \begin{tikzpicture}[
      x=1cm, y=1cm,
      box/.style={draw=black, line width=0.6pt,
                  minimum size=\skylineboxsize cm, inner sep=0pt},
      baseline=(current bounding box.south),
  ]
    \foreach \Pairs [count=\xi from 0] in {#1} {%
      \foreach \lbl/\col [count=\yi from 1] in \Pairs {%
        \pgfmathsetmacro{\bx}{\xi*(\skylineboxsize+\skylinegap)}%
        \pgfmathsetmacro{\by}{(\yi-1)*(\skylineboxsize+\skylinegap)}%
        \node[box, fill=\col!30, name=n-\xi-\yi] at (\bx,\by) {\small\lbl};
        }%
    }%
  \end{tikzpicture}%
}

\newcommand{\fcskylineOverlay}[2]{%
  \begin{tikzpicture}[
      x=1cm, y=1cm,
      box/.style={draw=black, line width=0.6pt,
                  minimum size=\skylineboxsize cm, inner sep=0pt},
      baseline=(current bounding box.south),
  ]
    \foreach \Pairs [count=\xi from 0] in {#1} {%
      \foreach \lbl/\col [count=\yi from 1] in \Pairs {%
        \pgfmathsetmacro{\bx}{\xi*(\skylineboxsize+\skylinegap)}%
        \pgfmathsetmacro{\by}{(\yi-1)*(\skylineboxsize+\skylinegap)}%
        \node[box, fill=\col!30, name=n-\xi-\yi] at (\bx,\by) {\small\lbl};
      }%
    }%
    #2
  \end{tikzpicture}%
}

\def\figA{%panyushev-reduction
\begin{gather*}
\raisebox{-.5\height}{\ShapeOverlay[.7]{%
1/11, 2/10, 3/9, 4/8, 5/7, 
7/1, 6/2, 5/3,
11/8, 10/9}{%
    1/1, 1/2, 1/3, 1/4, 1/5, 1/6, 1/7, 1/8, 1/9, 1/10,     
    2/1, 2/2, 2/3, 2/4, 2/5, 2/6, 2/7, 2/8, 2/9,       2/11,
    3/1, 3/2, 3/3, 3/4, 3/5, 3/6, 3/7, 3/8,      3/10, 3/11,
    4/1, 4/2, 4/3, 4/4, 4/5, 4/6, 4/7,      4/9, 4/10, 4/11,
    5/1, 5/2,      5/4, 5/5, 5/6,      5/8, 5/9, 5/10, 5/11,
    6/1,      6/3, 6/4, 6/5, 6/6, 6/7, 6/8, 6/9, 6/10, 6/11,
         7/2, 7/3, 7/4, 7/5, 7/6, 7/7, 7/8, 7/9, 7/10, 7/11,
                                       8/8, 8/9, 8/10, 8/11,
                                       9/8, 9/9, 9/10, 9/11,
                                      10/8,     10/10,10/11,
                                           11/9,11/10,11/11}{%
    8/1, 8/2, 8/3, 8/4, 8/5, 8/6, 8/7,
    9/1, 9/2, 9/3, 9/4, 9/5, 9/6, 9/7,
    10/1, 10/2, 10/3, 10/4, 10/5, 10/6, 10/7,
    11/1, 11/2, 11/3, 11/4, 11/5, 11/6, 11/7}{%
    \draw[myblue, line width=1pt, dashed] 
        (.5,-.5) -- (4.5, -.5) -- (4.5,-4.5) -- (7.5,-4.5)
                 -- (7.5,-7.5) -- (.5,-7.5) -- cycle;%
    }}
\quad
\longmapsto
\quad
\raisebox{-.5\height}{\Shape[.7]{%
1/4, 2/3, 5/7, 7/1, 6/2, 5/3}{%
    1/1, 1/2, 1/3, 
    2/1, 2/2, 2/4, 
    3/1, 3/2, 3/3, 3/4,
    4/1, 4/2, 4/3, 4/4, 
    5/1, 5/2, 5/4, 5/5, 5/6, 
    6/1, 6/3, 6/4, 6/5, 6/6, 6/7, 
    7/2, 7/3, 7/4, 7/5, 7/6, 7/7}{%
    1/5, 1/6, 1/7, 
    2/5, 2/6, 2/7,
    3/5, 3/6, 3/7,
    4/5, 4/6, 4/7%
    }}
\\
\raisebox{-.33\height}{\maxmeander[.6]{7/4}} 
    \quad \leadsto \quad
\raisebox{-.3\height}{\maxmeanderRetracting[.35]{7/4}{35}{.4}{\tiny}}
    \quad \leadsto \quad
\raisebox{-.5\height}{\maxmeanderLower[.6]{4/3}}
\end{gather*}
}

\def\figB{%panyushev_move
\begin{tabular}{@{}c@{}c@{}}
\maxmeander[.52]{2/3}
&
\raisebox{.25\height}{%
\gammaS[.52]{0/4/1, 0/6/2, 0/2/5, 2/4/4, 0/0/3}{1/5, 2/4, 2/1, 5/3}}
\end{tabular}
\quad $\leadsto$ \quad 
\begin{tabular}{c@{}c}
\scalebox{.65}{\MaxmeanderLower[.5]{3/5}{0}}
&
\scalebox{.65}{\raisebox{.6\height}{%
\coloredgammaS[.5]{0/4/1/white, 0/6/2/white, 0/2/5/white, 2/4/4/white, 0/0/3/white, 2/2/4'/red!40, 2/0/5'/red!40, 2/-2/3'/red!40}{1/5, 2/4, 2/1, 5'/3', 5/5', 4/4', 3/3'}}}
\end{tabular}
\quad $\leadsto$ \quad 
\begin{tabular}{c@{}c}
\maxmeander[.55]{5/3}
&
\raisebox{.35\height}{%
\coloredgammaS[.52]{0/8/4/white, 0/6/5/white, 0/4/1/white, 2/6/2/white, 0/2/3/white, 2/2/8/red!40, 2/4/7/red!40, 2/0/6/red!40}{4/5, 5/1, 1/8, 8/6, 4/2, 2/7, 3/6}}
\end{tabular}}

\def\figC{%orderideal_poset
\fenceposet[.8]{-1,1}
\hskip5em
\raisebox{.6\height}{%
\orderideals[.5]{1}{1.5/4/12/{1,2}, 0/6/123/{1,2,3}, 0/2/2/{2}, -1.5/4/23/{2,3}, -1.5/0/e/{\emptyset}}{2/e, 12/2, 23/2, 123/12, 123/23}}
\raisebox{-.2\height}{%
\begin{tikzpicture}[scale=.5]
\node at (0,6) {\color{red}$1$};
\node at (0,4) {\color{red}$2$};
\node at (0,2) {\color{red}$1$};
\node at (0,0) {\color{red}$1$};
\end{tikzpicture}}
\hskip5em
\raisebox{2.2\height}{%
${\color{red}1}t^3 \ + \ {\color{red}2}t^2 \ + \ {\color{red}1}t \ + \ 
{\color{red}1}$}
}

\def\figD{%fenceposet_move
\begin{tabular}{c@{\quad\quad}c}
\fenceposet[.8]{-1,1}
&
\raisebox{0\height}{%
\coloredorderideals[.5]{1}{1.5/4/12/{1,2}/red!40, 0/6/123/{1,2,3}/white, 0/2/2/{2}/red!40, -1.5/4/23/{2,3}/white, -1.5/0/e/{\emptyset}/red!40}{2/e, 12/2, 23/2, 123/12, 123/23}}
\end{tabular}
\quad\quad $\xmapsto{\ \ \ \pD \ \ \ }$ \quad\quad
\raisebox{-.2\height}{%
\begin{tabular}{c@{\quad\quad}c}
\fenceposet[.8]{-1,1,-1}
&
\raisebox{-.2\height}{%
\coloredorderideals[.5]{1}{1.5/4/124/{1,2,4}/white, 0/6/1234/{1,2,3,4}/white, 0/2/24/{2,4}/white, -1.5/4/234/{2,3,4}/white, -1.5/0/4/{4}/white, 3/2/12/{1,2}/red!40, 1.5/0/2/{2}/red!40, 0/-2/e/{\emptyset}/red!40}{1234/234, 1234/124, 2/e, 12/2, 234/24, 124/12, 124/24, 4/e, 24/4, 24/2}}
\end{tabular}}
}

\def\figE{%panyushev_treeleft_trimmed
\begin{tikzpicture}[level 1/.style={sibling distance=22em, level distance=15em}]
    \coordinate (sw) at (-10,2);
    \coordinate (ne) at (7.5,-9.8);
    \begin{scope}
       \clip  (sw) rectangle (ne); % clip some region
\node {% new root (formerly right child)
$\begin{array}{c@{\ \ }c}
\begin{array}{c}
\maxmeander[.65]{2/5}
\\
\dfrac{{t}^4 + {t}^3}{{t}^3 + 2{t}^2 + {t} + 1}%}
\end{array}
&
\raisebox{-.2\height}{%
\gammaS[.45]{2/8/2, 0/6/1, 0/4/7, 0/2/3, 0/0/5, 2/4/4, 2/6/6}{2/1, 1/7, 7/3, 3/5, 2/6, 6/4}}
\qquad\qquad\qquad\qquad \end{array}$}
child {
node {%left child
$\begin{array}{c@{\ \ }c}
\dfrac{{t}^5 + {t}^4}{{t}^4 + 2{t}^3 + 2{t}^2 + {t} + 1}
&
\raisebox{.5\height}{% 
\coloredgammaS[.45]{2/8/2/blue!20, 0/6/1/blue!20, 2/6/8/white, 
0/4/9/white, 2/4/4/white, 
0/2/3/white, 2/2/6/white, 
0/0/7/white, 
0/-2/5/white}{2/1, 1/9, 9/3, 3/7, 7/5, 2/8, 8/4, 4/6}}
\qquad\qquad\qquad\end{array}$
}
edge from parent[->, thick,blue]
node[left, xshift=2em, yshift=1em] {$\raisebox{-.2\height}{\MaxmeanderLower[.3]{7/2}{1}} \quad {\pL}$}
}
child {node {%right child
$\begin{array}{c@{\ \ }c}
\dfrac{{t}^5 + 2{t}^4 + 2{t}^3 + {t}^2 + {t}}{{t}^3 + 2{t}^2 + {t} + 1}
&
\raisebox{.5\height}{%
\coloredgammaS[.45]{2/10/6/white, 0/8/7/white, 2/8/2/white, 
0/6/1/white, 2/6/4/white, 
0/4/5/white, 
0/2/3/white, 
4/6/11/red!40, 4/4/9/red!40, 2/4/12/red!40, 2/2/8/red!40, 0/0/10/red!40}{6/2, 6/7, 7/1, 5/3, 2/11, 4/9, 11/9, 1/12, 12/8, 5/8, 3/10}}
\end{array}\qquad\qquad\qquad\qquad\qquad\qquad$
}
edge from parent[->, thick,red, pos=.55]
node[right, xshift=-2.25em, yshift=1.2em] {${\pR} \quad \ \ \raisebox{-.25\height}{\MaxmeanderLower[.3]{5/7}{0}}$}
};
\end{scope}
\end{tikzpicture}
}

\def\figF{%fenceposet_treetrimmed
\begin{tikzpicture}[level 1/.style={sibling distance=20em, level distance=15em}]
    \coordinate (sw) at (-11.6,2.4);
    \coordinate (ne) at (7.4,-9.9);
    \begin{scope}
       \clip  (sw) rectangle (ne); % clip some region
\node {% new root (formerly right child)
$\begin{array}{c@{\!\!}c}
\begin{array}{c}
\fenceposet[.7]{-1,1,-1}
\\[-.5ex]
\dfrac{{t}^4 + 2{t}^3 + {t}^2 + {t}}{{t}^2 + {t} + 1}%}
\end{array}
&
\raisebox{0.1\height}{%
\coloredorderideals[.5]{1}{1.5/4/124/{1,2,4}/white, 0/6/1234/{1,2,3,4}/white, 0/2/24/{2,4}/white, -1.5/4/234/{2,3,4}/white, -1.5/0/4/{4}/white, 3/2/12/{1,2}/red!40, 1.5/0/2/{2}/red!40, 0/-2/e/{\emptyset}/red!40}{1234/234, 1234/124, 2/e, 12/2, 234/24, 124/12, 124/24, 4/e, 24/4, 24/2}}
\qquad\qquad\qquad\qquad \end{array}$}
child {
node {%left child
$\begin{array}{c@{\!\!\!\!\!\!}c}
\dfrac{{t}^5 + 2{t}^4 + {t}^3 + {t}^2}{{t}^4 + 2{t}^3 + 2{t}^2 + 2{t} + 1}
&
\raisebox{.6\height}{% 
\coloredorderideals[.5]{1}{-2.25/6/1245/{1245}/blue!20, -3.75/8/12345/{12345}/blue!20, -3.75/4/245/{245}/blue!20, -5.25/6/2345/{2345}/blue!20, -5.25/2/45/{45}/blue!20, 1.5/4/124/{124}/white, 0/6/1234/{1234}/white, 0/2/24/{24}/white, -1.5/4/234/{234}/white, -1.5/0/4/{4}/white, 3/2/12/{12}/white, 1.5/0/2/{2}/white, 0/-2/e/{\emptyset}/white}{12345/1234, 12345/2345, 12345/1245, 1245/124, 1245/245, 45/4, 245/45, 2345/245, 2345/234, 245/24, 1234/234, 1234/124, 124/12, 124/24, 4/e, 24/4, 24/2, 2/e, 12/2, 234/24}}
\qquad\qquad\qquad\qquad\quad\end{array}$
}
edge from parent[->, thick,blue]
node[pos=0.45, xshift=-1.2em]
{$\tikz[baseline] \node[anchor=center, inner sep=0pt]
{\hspace{-0.6em}\fenceposetExtended[.37]{-1,1,-1}{1}};\qquad\raisebox{-1.0\height}{$\pU$}$%
}
}
child {node {%right child
$\begin{array}{c@{\!\!\!}c}
\dfrac{{t}^5 + 2{t}^4 + 2{t}^3 + 2{t}^2 + {t}}{{t}^2 + {t} + 1}
&
\raisebox{.6\height}{%
\coloredorderideals[.5]{1}{0/6/12345/{12345}/white, 1.5/4/1245/{1245}/white, 0/2/245/{245}/white, -1.5/4/2345/{2345}/white, -1.5/0/45/{45}/white, 3/2/125/{125}/white, 1.5/0/25/{25}/white, 0/-2/5/{5}/white, 4.5/0/12/{12}/red!40, 3/-2/2/{2}/red!40, 1.5/-4/e/{\emptyset}/red!40}{12345/2345, 12345/1245, 2345/245, 1245/245, 1245/125, 125/25, 125/12, 245/25, 245/45, 12/2, 25/5, 45/5, 25/2, 2/e, 5/e}}
\qquad\qquad\qquad\quad\end{array}$
}
edge from parent[->, thick,red]
node[pos=0.35, xshift=1.2em]
{$\raisebox{-1.5\height}{$\pD$}\quad\tikz[baseline] \node[anchor=center, inner sep=0pt]
{\hspace{0.6em}\fenceposetExtended[.37]{-1,1,-1}{-1}};$%
}
};
\end{scope}
\end{tikzpicture}
}

\def\figG{%three_triples
\fcskyline{{/white,/red},{/white,/red, /red},{/white,/red,/white,/white}}
\quad\quad
\fcskyline{{/white,/red,/red},{/white,/red,/red,/white},{/white,/red,/white}}
\quad\quad\quad\quad
\fcskylineOverlay{%
{7/white},
{4/white,17/red},
{6/white,20/red,10/white},
{1/white,14/red,5/white,8/white,18/red},
{3/white,16/red,9/white,19/red,23/red},
{2/white,11/red,15/red,21/red},
{13/red,22/red},
{12/red}}{%
\tikzset{myedge/.style={very thick, shorten <=-.5pt,shorten >=-.5pt,blue!70}}
\draw[myedge] (n-0-1) -- (n-1-1);
\draw[myedge] (n-1-1) -- (n-2-1);
\draw[myedge] (n-2-1) -- (n-3-1);
\draw[myedge] (n-3-1) -- (n-4-1);
\draw[myedge] (n-4-1) -- (n-5-1);
\draw[myedge] (n-5-1) -- (n-6-1);
\draw[myedge] (n-6-1) -- (n-7-1);
\draw[myedge] (n-1-2) -- (n-2-2);
\draw[myedge] (n-2-2) -- (n-3-2);
\draw[myedge] (n-3-2) -- (n-4-2);
\draw[myedge] (n-4-2) -- (n-5-2);
\draw[myedge] (n-5-2) -- (n-6-2);
\draw[myedge] (n-2-3) -- (n-3-3);
\draw[myedge] (n-3-3) -- (n-4-3);
\draw[myedge] (n-4-3) -- (n-5-3);
\draw[myedge] (n-3-4) -- (n-4-4);
\draw[myedge] (n-4-4) -- (n-5-4);
\draw[myedge] (n-3-5) -- (n-4-5);}
}

\def\figH{%building-a-filling
    \begin{tikzpicture}[x=0.6cm,y=0.6cm]
        \node[anchor=south west] (P) at (0,0) {%
        \fcskylineOverlay{%
{\sfA/white},
{\sfA/white,\sfB/red},
{\sfA/white,\sfB/red,\sfA/white},
{\sfA/white,\sfB/red,\sfA/white,\sfA/white,\sfB/red},
{\sfA/white,\sfB/red,\sfA/white,\sfB/red,\sfB/red},
{/gray,/gray,/gray,/gray},
{/gray!40!white,/gray!40!white},
{/gray!40!white}}{%
\tikzset{myedge/.style={very thick, shorten <=-.5pt,shorten >=-.5pt,blue!70}}
\draw[myedge] (n-0-1) -- (n-1-1);
\draw[myedge] (n-1-1) -- (n-2-1);
\draw[myedge] (n-2-1) -- (n-3-1);
\draw[myedge] (n-3-1) -- (n-4-1);
\draw[myedge] (n-1-2) -- (n-2-2);
\draw[myedge] (n-2-2) -- (n-3-2);
\draw[myedge] (n-3-2) -- (n-4-2);
\draw[myedge] (n-4-2) -- (n-5-2);
\draw[myedge] (n-2-3) -- (n-3-3);
\draw[myedge] (n-3-3) -- (n-4-3);
\draw[myedge] (n-3-4) -- (n-4-4);
\draw[myedge] (n-4-4) -- (n-5-4);
\draw[myedge] (n-3-5) -- (n-4-5);}
};
\node (col0) at (0.6,-.2) {\tiny$\color{orange}\sf 0$};
\node (col1) at (1.6,-.2) {\tiny$\color{orange}\sf 1$};
\node (col2) at (2.6,-.2) {\tiny$\color{orange}\sf 2$};
\node (col3) at (3.6,-.2) {\tiny$\color{orange}\sf 3$};
\node (col4) at (4.6,-.2) {\tiny$\color{orange}\sf 4$};
\node (col5) at (5.6,-.2) {\tiny$\color{orange}\sf 5$};
\node (col6) at (6.6,-.2) {\tiny$\color{orange}\sf 6$};
\node (col7) at (7.6,-.2) {\tiny$\color{orange}\sf 7$};
\node (row0) at (-.2,0.6) {\tiny$\color{orange}\sf 1$};
\node (row1) at (-.2,1.6) {\tiny$\color{orange}\sf 2$};
\node (row2) at (-.2,2.6) {\tiny$\color{orange}\sf 3$};
\node (row3) at (-.2,3.6) {\tiny$\color{orange}\sf 4$};
\node (row4) at (-.2,4.6) {\tiny$\color{orange}\sf 5$};
    \end{tikzpicture}
}

\def\figI{%dim_L23,41
\eigenShapeOverlay[.7]{%
  1/1/\bl{0},    1/2/\bgr{1},%
  2/1/\rd{-1},   2/2/\bl{0},%
  3/1/0,         3/2/\bgr{1},  3/3/\bl{0},   3/4/\rd{-1},  3/5/\bgr{1},%
  4/1/\bgr{1},   4/2/2,        4/3/1,        4/4/\bl{0},   4/5/2,%
  5/5/\bl{0}%
}{%
  1/3, 1/4, 1/5,%
  2/3, 2/4, 2/5,%
  5/1, 5/2, 5/3, 5/4%
}{%
\fill[magenta!90, nearly transparent] (3,-4) circle (.55);
}
\quad\quad\quad
\raisebox{-.55\height}{\meanderDUnderOverlay[.9]{5}{1/2, 4/1, 3/5, 3/2}{}{%
\draw[magenta!80, nearly transparent, line width=3.25pt, shorten >=-.5pt] (4) to[bend left=40+5*3] (1);
\draw[magenta!80, nearly transparent, line width=3.25pt, shorten >=-.5pt] (1) to[bend left=40+5*1] (2);
\draw[magenta!80, nearly transparent, line width=3.25pt, shorten >=-.5pt] (3) to[bend left=40+5*1] (2);
\fill[magenta!80, nearly transparent] (4,0) circle (.27);
\fill[magenta!80, nearly transparent] (3,0) circle (.27);
}}
\quad
\gammaSUnderOverlay[.7]{%
  0/0/2,  2/0/5,%
  0/2/1,  2/2/3,%
  0/4/4%
}{%
  4/1, 1/2, 3/2, 3/5%
}{%
}{%
\fill[magenta!80, nearly transparent] (0,4) circle (.52);
\fill[magenta!80, nearly transparent] (2,2) circle (.52);
}%
\quad\quad\quad
\begin{tabular}{c|rrrr}
$\lambda$   & $-1$ & $0$ & $1$ & $2$ \\
\hline
$d_\lambda$ &    2 &   5 &   5 &   2
\end{tabular}
}

\newcommand{\demph}[1]{\textcolor{myblue}{\emph{#1}}}
\renewcommand{\demph}[1]{\emph{#1}}

\def\ba{\mathbf{a}}
\def\bb{\mathbf{b}}
\def\bc{\mathbf{c}}
\def\be{\mathbf{e}}

\newcommand{\bs}[1]{\boldsymbol{#1}}

\def\fg{\mathfrak g} % a Lie algebra
\def\fp{\mathfrak p} % a (bi)parabolic subalgebra
\def\cE{\mathcal E} % edges
\def\cJ{\mathcal J}

\def\ZZ{\mathbb Z}

\def\CC{\mathbb C}
\def\sl{\mathfrak{sl}}
\def\gl{\mathfrak{gl}}
\def\L{\mathfrak L}
\def\valpha{\boldsymbol{\alpha}}
\def\vbeta{\boldsymbol{\beta}}

\def\Ooms{\Lambda} % (Ooms spectrum)

\def\hatF{\widehat{F}}

\newcommand{\pab}[1][a,b]{\fp(#1)}
\newcommand{\puab}[1][a,b]{\fp\underline{(#1)}}
\def\sln{\sl_n}
\def\gln{\gl_n}
\newcommand{\Top}[1]{#1^{+}}
\newcommand{\Bot}[1]{#1^{-}}
\newcommand{\Jfull}[1]{J_{#1}}
\newcommand{\Jtop}[1]{J_{#1}^+}
\newcommand{\Jbot}[1]{J_{#1}^-}
\newcommand{\Rfull}[1]{R_{#1}}
\newcommand{\Rtop}[1]{R_{#1}^+}
\newcommand{\Rbot}[1]{R_{#1}^-}

\def\dimvec{\mathop{\underline{\mathrm{dim}}}}

\def\ad{\operatorname{ad}}
\def\rnk{\operatorname{rnk}}
\def\col{\operatorname{col}}

\def\qand{\quad\hbox{and}\quad}

\newtheorem{thm}{Theorem}[section]

\def\pL{\mathsf L}
\def\pR{\mathsf R}
\def\pU{\mathsf U}
\def\pD{\mathsf D}

\newtheorem{theorem}{Theorem}%[section]
\newtheorem{corollary}[theorem]{Corollary}
\newtheorem{lemma}[theorem]{Lemma}
\newtheorem{proposition}[theorem]{Proposition}

\theoremstyle{definition}
\newtheorem{remark}[theorem]{Remark}
\newtheorem{example}[theorem]{Example}%[section]

\title{Fence Posets, Good Gradings and Frobenius Maximal Parabolics}

\author{Anthony Giaquinto}
\address{Department of Mathematics and Statistics,
		Loyola University Chicago, Chicago, IL 60660 USA}
\email{agiaqui@luc.edu}
\author{John Irving}
\address{Department of Mathematics and Computing Science, Saint Mary's University, 923 Robie St, Halifax, NS, B3H3C3 Canada}
\email{john.irving@smu.ca}
\author{Aaron Lauve}
\address{Department of Mathematics and Statistics,
		Loyola University Chicago, Chicago, IL 60660 USA}
\email{alauve@luc.edu}
\author{Mitja Mastnak}
\address{Department of Mathematics and Computing Science, Saint Mary's University, 923 Robie St, Halifax, NS B3H3C3 Canada}
\email{mitja.mastnak@smu.ca}
\thanks{Mitja Mastnak was supported in part by NSERC (Canada).}



\begin{document}

\begin{abstract}
Let $\L$ be a Frobenius maximal parabolic subalgebra of $\sln$. For any
$F\in\L^*$ for which the Kirillov form
$B_F(x,y)=F([x,y])$ is non-degenerate, let $\widehat F$ denote the
associated principal element. We prove that the multiplicities of the
eigenvalues of $\operatorname{ad}_{\widehat F}$ on $\L$ form a
unimodal sequence symmetric about $\frac12$. We also prove that the multiplicities of the
eigenvalues of $\operatorname{ad}_{\widehat F}$ on $\gln$ form a
unimodal sequence symmetric about $0$. The proof relates the ranked meander associated to $\L$ to the order ideals of a related fence poset through Panyushev reduction. The known unimodality of the rank polynomial of the fence poset implies that of the meander, which in turn determines a good grading of $\gln$ in the sense of Elashvili and Kac. We prove that this grading coincides with that induced by the principal element and that the pyramid associated to this grading may be filled in such a way that its good element $e$ lies in $\L$. The two unimodality results then follow from the injectivity properties of $\ad_e$ coming from the good grading and the duality induced by the bilinear form $B_F$.
\end{abstract}

\maketitle

%%%%%%%%%%%%%%%%%%%%%%%%%%%%
\section{Introduction}
\label{sec:intro}
%%%%%%%%%%
	Let $\L$ be a finite-dimensional Lie algebra over a field of characteristic zero. For $F\in \L^*$, the Kirillov form is the skew-symmetric bilinear form on $\L$ given by $B_F(x,y)=F([x,y])$. The index of $\L$ is the minimum dimension of the kernel of $B_F$ as $F$ ranges over $\L^*$. The Lie algebra $\L$ is \emph{Frobenius} if the index is zero or, equivalently, if its coadjoint representation has an open orbit in $\L^*$. In this case, every $F$ in this orbit determines a non-degenerate form $B_F$, and such $F$ are called Frobenius functionals. Thus Frobenius Lie algebras are an important class in the study of the index and the geometry of coadjoint orbits. The terminology ``Frobenius Lie algebra'' was introduced by Ooms in
	\cite{ooms1980frobenius}. Simple Lie algebras are never Frobenius as the index is equal to the rank, while the existence of the non-degenerate form $B_F$ gives Frobenius Lie algebras additional structure not present  in arbitrary non-semisimple algebras. 
	
	An important source of examples lies within the class of \emph{biparabolic} subalgebras of simple Lie algebras. In type $A$ biparabolic subalgebras of $\sln$ are also called \emph{seaweed} algebras, and we use the terms {biparabolic} and {seaweed} interchangeably throughout. In \cite{dergachev2000index}, a simple combinatorial formula for the index of a type $A$ seaweed Lie algebra was given in terms of its meander graph. Moreover, the meander determines an explicit functional realizing the index and so in the Frobenius case it provides a non-degenerate form $B_F$. The meander construction is closely related to Kostant's cascade of strongly
	orthogonal roots, which plays an important role in the study of biparabolic subalgebras of simple Lie algebras, see \cite{panyushev2023combinatorial}.
    	 
	For a Frobenius functional $F$, the non-degeneracy of $B_F$ provides an isomorphism $\L \simeq \L^*$ and thus there is a distinguished element $\hatF \in \L$ corresponding to $F$ under this isomorphism. The element $\hatF$ was first studied by Ooms 
	\cite{ooms1980frobenius}, who remarked that $\hatF$ and the spectrum of $\ad_{\hatF}$ should play an important role 
	in the study of $\L$. In \cite{panyushev2023combinatorial}, $\hatF$ is 
	called the Ooms element of $\L$, but we follow the terminology of 
	\cite{gerstenhaber2009principal} and call it the {\it principal element} 
	of $\L$. Ooms showed in \cite{ooms1980frobenius} that, for an arbitrary Frobenius Lie algebra,
	the eigenvalues of $\ad_{\hatF}$, counted with
	multiplicity, are symmetric about $\frac12$. In general, however, the
	eigenvalues need not be integers. For Frobenius biparabolic subalgebras
	of $\sln$, the spectrum has a considerably richer structure. Gerstenhaber and Giaquinto showed in \cite{gerstenhaber2009principal} that the principal element
	is semisimple and that the eigenvalues of $\ad_{\hatF}$ on $\gln$, and hence on $\L$, are integers and can be computed directly from
	the meander. More generally, Joseph proved that the integrality property holds for
every Frobenius biparabolic subalgebra of a semisimple Lie algebra
\cite{joseph2015integrality}. (In his terminology, the pair
$(\hatF,F)$ arising here is an ``adapted pair''.) Besides being integral in the case of seaweed algebras, Coll, Hyatt and Magnant showed in \cite{coll2018unbroken} that these eigenvalues form an unbroken string of integers. 
	
	As the eigenvalues of a Frobenius seaweed algebra are integers centered around $\frac{1}{2}$, it is natural to ask how these are distributed. In a 2009 email to Gerstenhaber and Giaquinto, Duflo asked whether the multiplicities of the eigenvalues of
	a principal element of a Frobenius seaweed algebra are unimodal. The main result of this paper is the following theorem.
	
	\begin{thm}\label{mod-conj}
		Suppose $\L\subseteq \sln$ is maximal parabolic and Frobenius with principal
		element $\hatF$. 
		\begin{enumerate}
			\item The multiplicities of the eigenvalues of $\ad_{\hatF}$ on $\L$
			form a unimodal sequence around $\frac12$. 
			\item The multiplicities of the eigenvalues of $\ad_{\hatF}$ on
			$\gln$ form a unimodal sequence around $0$.
		\end{enumerate}
	\end{thm}

Our approach (developed in Sections \ref{sec:panyushev} and \ref{sec:fence posets}) is inductive and relies heavily on the rank polynomial of the meander, whose coefficients count the number of its vertices at each rank. The reduction process is the Panyushev reduction introduced in \cite{panyushev2001inductive} that inductively computes the index of biparabolic subalgebras of $\sln$. When applied to the Frobenius maximal parabolic case, this recursive process surprisingly agrees with a related recursion for the \emph{order ideals of fence posets}. The rank polynomials of these order ideal posets were conjectured to be unimodal by Morier-Genoud and Ovsienko in \cite{morier2020q-deformed}. This conjecture was recently proved by Oguz and Ravichandran, see \cite{kantarci2023rank}, and so the rank polynomial of the Frobenius maximal parabolic meander is also unimodal.

The unimodality of the rank polynomial of the meander allows us to invoke the theory of good gradings of $\gln$ in the sense of Elashvili and Kac, see \cite{elashvili2005classification}. A \emph{good grading} on $\gln$ comes equipped with a nilpotent element $e\in \gln$ such that $\ad_e$ acts injectively on negative weight spaces of the grading. The element $e$ is constructed from an Elashvili--Kac pyramid corresponding to the rank polynomial of the meander. The pyramid provides the grading of $\gln$, and we prove that this grading coincides with that given by the principal element. The desired unimodality result for $\gln$ then follows immediately. 

To prove the unimodality on a maximal parabolic Frobenius subalgebra of $\sln$, we show (in Section~\ref{sec:good gradings}) that the pyramid admits a filling such that the associated element $e$ lies in the maximal parabolic subalgebra in question. Hence, the injectivity of $\ad_e$ on all of the negative weight spaces of $\gln$ restricts to injectivity on those of a maximal parabolic subalgebra. This, combined with the duality coming from the bilinear form $B_F$, establishes the desired unimodality result for maximal parabolic Frobenius subalgebras of $\sln$.

This paper is organized as follows. In Section~\ref{sec:frobenius} we establish notation and recall basic definitions and results about Frobenius biparabolic subalgebras of $\sln$, their meanders and their principal elements. We show how the eigenvalues of a principal element can be read off from the meander. In Section~\ref{sec:panyushev} we specialize to maximal parabolics and recall the Panyushev reduction process and its effect on their meanders. Section~\ref{sec:fence posets} introduces the necessary facts about fence posets and their order ideals, and shows that the rank polynomial of a Frobenius maximal parabolic meander coincides with the rank polynomial of the order ideal poset of an associated fence poset. In Section~\ref{sec:good gradings}, we use this connection together with the Elashvili--Kac theory of good gradings of $\gln$ to prove  Theorem \ref{mod-conj} . Finally, in Section~\ref{sec:further remarks} we discuss several further phenomena for Frobenius biparabolic subalgebras.

%%%%%%%%%%%%%%%%%%%%%%%%%%%%
\section{Frobenius Biparabolic Lie Algebras and Meanders}
\label{sec:frobenius}
%%%%%%%%%%

\subsection{Biparabolic subalgebras and meanders}
\label{subsec:biparabolic-meanders}
%%%%%%%%%%
We retain the terminology concerning Frobenius Lie algebras from the introduction. We now recall the type $A$ biparabolic, or seaweed, subalgebras of $\sln$ and
their associated meander graphs. Following 
\cite[Definition 4.1]{dergachev2000index}, a biparabolic, or seaweed,
subalgebra of $\sln$ is determined by two compositions
\[
\valpha=(\alpha_1,\ldots,\alpha_r)
\qquad\text{and}\qquad
\vbeta=(\beta_1,\ldots,\beta_s)
\]
of $n$ and defined as follows. Let $e_1,\ldots, e_n$ be the standard basis of $\CC^n$ and set
\[V_i = \operatorname{span} (e_1,\ldots,e_{\beta_1+\cdots+\beta_i}) \quad \text{and}\quad W_j=\operatorname{span}
(e_{\alpha_1+\cdots+\alpha_j+1},\ldots,e_n).
\]
The subalgebra of $\sln$ preserving $V_i$ and $W_j$ for all $i=1,\ldots, s-1$ and $j=1,\ldots, r-1$ is a \demph{biparabolic} subalgebra and is denoted $\L(\valpha\mid\vbeta)$. A basis-free description of biparabolic subalgebras can be found in \cite{panyushev2001inductive}.

Dergachev and Kirillov associate to each $\L(\valpha\mid\vbeta)$ a certain directed graph, the \demph{meander}, which we now describe. The vertices of the meander are points $1,\ldots, n$ on a horizontal line from left to right. Partition the vertices into  blocks of sizes $\alpha_1,\ldots, \alpha_r$. Within each block, create non-crossing arcs above the line directed left to right connecting its first vertex to its last, its second to its second to last, etc. Next take a second partition of the vertices into blocks of size $\beta_1,\ldots, \beta_s$. Within these blocks create non-crossing arcs below the line from right to left connecting its last vertex to its first, its second to last to its second, etc. If any $\alpha_i$ or $\beta_j$ is odd, its middle vertex remains unpaired on the top or bottom, respectively. The resulting graph is the meander, denoted $M(\valpha\mid\vbeta)$. Let $\cE$ denote the set of edges in $M(\valpha\mid\vbeta)$. If $s=(i,j)$ is an edge in $M(\valpha\mid\vbeta)$ from $i$ to $j$, define $\be_s$ to be the standard matrix unit $\be_{i,j}$. Dergachev and Kirillov define the meander functional by
\[
F_{DK}
=
\sum_{s\in\cE}\be_s^*
\]
and prove that it realizes the index of $\L$, in the sense that the kernel
of the associated Kirillov form has minimal dimension.

Every component of the meander, viewed as an {undirected} graph, is a cycle, a path or an isolated point. See Figure \ref{fig:biparabolic-not-frobenius} for an illustration.  
\begin{figure}[!hbt]
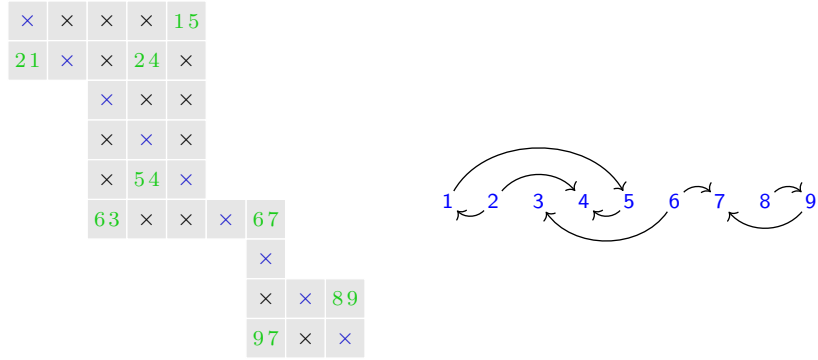

\centering
	\raisebox{-.3\height}{\Shape[.7]{1/5, 2/4, 2/1, 6/3, 5/4, 6/7, 9/7, 8/9}{%
    1/1, 1/2, 1/3, 1/4,  
         2/2, 2/3,      2/5,
              3/3, 3/4, 3/5, 
              4/3, 4/4, 4/5, 
              5/3,      5/5, 
                   6/4, 6/5, 6/6,      
                                  7/7,
                                  8/7, 8/8,     
                                       9/8, 9/9}{}}
	\quad
    \quad
	\meanderD[.8]{9}{1/5, 2/4, 6/7, 8/9, 2/1, 6/3, 5/4, 9/7}
\caption{The biparabolic $\L((5,2,2) \mid (2,4,3))$ (at left) and its meander graph.}
\label{fig:biparabolic-not-frobenius}
\end{figure}

\begin{thm}[Dergachev--Kirillov {\cite[Theorem~5.1]{dergachev2000index}}]
\label{th:index-from-meander}
Let $M(\valpha\mid\vbeta)$ be the meander associated to the biparabolic subalgebra $\L(\valpha\mid\vbeta)\subseteq\mathfrak{sl}_n$. Then the index of $\L(\valpha\mid\vbeta)$ is $2c+p-1$, where $c$ is the number of cycles in $M(\valpha\mid\vbeta)$, and $p$ is the number of paths and isolated points in $M(\valpha\mid\vbeta)$.
\end{thm}
An immediate consequence of the theorem is that $\L(\valpha\mid\vbeta)$ is Frobenius if and only if its meander is a single path. From Figure \ref{fig:biparabolic-not-frobenius}, we see that the undirected meander for $\L((5,2,2)\mid(2,4,3))$ consists of one cycle and one path, so the index is $2$, and the algebra is not Frobenius.

\subsection{Principal elements and their spectra}
\label{subsec:principal-spectrum}
%%%%%%%%%%

For an arbitrary Frobenius Lie algebra $\L$ and a Frobenius 
functional $F$, the map
\[
\phi_F:\L\longrightarrow\L^*,
\qquad
\phi_F(x)(y)=B_F(x,y),
\]
is an isomorphism. Therefore,
\[
\hatF:=\phi_F^{-1}(F)\in\L
\]
is the unique element satisfying
\[
F(y)=F([\hatF,y])
\qquad\text{for all }y\in\L.
\]
Ooms showed in \cite{ooms1980frobenius} that the eigenvalues of $\ad_{\hatF}$ and their multiplicities are independent of the choice of Frobenius functional. Thus the spectrum is an invariant of the Frobenius Lie algebra $\L$.

For the remainder of the paper, when $\L$ is a Frobenius biparabolic
subalgebra of $\sln$, we take $F=F_{DK}$ and write $\hatF$ for its associated
principal element.
We next recall the Gerstenhaber--Giaquinto construction of the principal
element and the resulting description of its eigenvalues
\cite{gerstenhaber2009principal}. 

Let $s=(i,j)$ be any edge of the meander $M$ corresponding to $\L$. Since $M$ is a single path, deleting $s$ disconnects $M$ into two components. Let $h_s$ be the sum of all elements
$\be_{k,k}-\frac{1}{n}I$ for which 
$k$ lies in the component containing $i$. Then
\[
[h_s,\be_t]=\delta_{s,t}\be_t
\qquad\text{for all }s,t\in\cE.
\]
Gerstenhaber and Giaquinto show that the elements $h_s$ are linearly
independent and that the principal element associated to $F$ is the sum
\[
\hatF=\sum_{s\in\cE}h_s.
\]

The usefulness of this description of $\hatF$ is that it translates the
computation of the eigenvalues of $\ad_{\hatF}$ directly to  paths in the meander $M$. Each directed edge contributes an eigenvalue $1$, and
reversing the direction changes the sign. Since consecutive edges of a
path correspond, via bracketing, to the matrix unit joining the
endpoints, the derivation property of $\ad_{\hatF}$ shows that the
eigenvalues add along the path. Thus every matrix unit $\be_{i,j}$ is an
eigenvector for $\ad_{\hatF}$, and if
\[
[\hatF,\be_{i,j}]=\lambda_{i,j}\be_{i,j},
\]
then $\lambda_{i,j}$ is the signed \emph{path weight} of the unique path from $i$ to
$j$ in $M$.  Equivalently, $\lambda_{i,j}$ equals the difference $\rnk(i)-\rnk(j)$, where $\rnk$ is any integer function on $[1,n]$  satisfying $\rnk(k)=\rnk(\ell)+1$ for each edge $(k,\ell)$ of $M$.  The unique such function that attains minimum value 0 is called the \emph{rank function} on $M$. Taken together, we refer to $M$ and its rank function as the \emph{ranked meander} associated with $\L$.

We illustrate the eigenvalue algorithm for $\L((2,3)\mid (4,1))$ in 
Figure~\ref{fig:eigentally}. First note that the undirected meander is a single path, and so the algebra is Frobenius. The integer appearing in row $i$ and column $j$ in the shape on the left is the eigenvalue $\lambda_{i,j}$. For example, the highlighted $1$ in row $4$ and column $3$ 
signifies that $\be_{4,3}$ is an eigenvector of $\ad_{\hatF}$ with 
eigenvalue $1$, which corresponds to the weight of the path from $4$ to $3$ in the meander graph or, equivalently, the  difference $\rnk(4)-\rnk(3)=2-1=1$ in the ranked meander. The full list of eigenvalues, together with their multiplicities, is given in the tally at the right of the figure. 

\begin{figure}[!hbt]
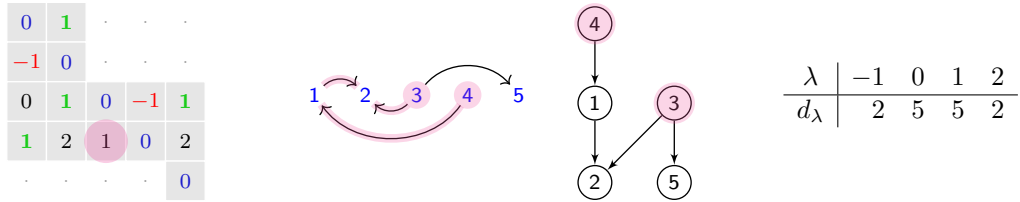

    \centering
    \figI
    \caption{Computing the eigenvalues and their multiplicities for
$\L((2,3)\mid (4,1))$.}
    \label{fig:eigentally}
\end{figure}
\subsection{Properties of the spectrum}
Let $\hatF$ be the principal element of a Frobenius biparabolic subalgebra $\L$ of $\sln$. For $\lambda \in \ZZ$, let $\L_\lambda$ and $\mathfrak{g}_{\lambda}$ be the $\lambda$-eigenspaces of 
$\ad_{\hatF}$ on $\L$ and $\gln$, respectively. Let 
\[\Ooms(\L) = \{\lambda\in\ZZ\mid \L_\lambda\neq\bs0\} \qquad \text{and}\qquad  \Ooms(\gln;\L) = \{\lambda\in\ZZ\mid \mathfrak{g}_\lambda\neq\bs0\}.\] 
There are then eigenspace decompositions of $\L$ and $\gln$ 
\[
\L=\bigoplus_{\lambda\in\Ooms(\L)}\L_\lambda \qquad \text{and}\qquad \gln=\bigoplus_{\lambda\in\Ooms(\gln;\L)}\mathfrak{g}_\lambda.
\]
The next lemma gives important elementary properties of the eigenspace decomposition of $\L$.
\begin{lemma}
\label{th:elementary-properties}
Suppose $\L$ is a Frobenius biparabolic subalgebra of $\sln$ with
functional $F$ and principal element $\hatF$. Then
\begin{enumerate}
\item $[\L_\lambda,\L_\mu]\subseteq \L_{\lambda+\mu}$.
\item If $x\in\L_\lambda$ with $\lambda\neq 1$, then $F(x)=0$.
\item The space $\L_\lambda$ is dual to $\L_{1-\lambda}$ under the
non-degenerate form $B_F$.
\end{enumerate}
\end{lemma}
\begin{proof}
Part (1) follows from the fact that $\ad_{\hatF}$ is a derivation of the Lie bracket $[-,-]$.
For Part (2), suppose $x\in \L_{\lambda}$ is nonzero. Then
\[F(x)=F([\hatF,x])=F(\lambda x)=\lambda F(x)\]
and so $F(x)=0$ unless $\lambda = 1$. The duality assertion of Part (3) follows  from the first two parts of the lemma and the non-degeneracy of the bilinear form $B_F$.
\end{proof}

Coll, Hyatt, and Magnant proved in \cite{coll2018unbroken} that the
spectrum of a Frobenius seaweed is an unbroken string of integers. If $d_\lambda=\dim\L_\lambda,$
then Lemma~\ref{th:elementary-properties} implies that $ d_\lambda=d_{1-\lambda}.$ Thus if $m$ is the largest eigenvalue of $\ad_{\hatF}$, then the unbroken 
spectrum property and the duality property imply that $\Ooms(\L) = \{-m+1,-m+2,\ldots,m-1,m\}$. Consequently the dimension vector
\[
\dimvec(\L)
=
(d_{-m+1},d_{-m+2},\ldots,d_{m-1},d_m)
\]
is symmetric around $\frac12$. 

Next let $\delta_{\lambda}= \dim \mathfrak{g}_{\lambda}$. The algorithm to compute the eigenvalues of $\ad_{\hatF}$ on $\gln$ shows that $\lambda_{i,j}=-\lambda_{j,i}$ for all $i$ and $j$ since the path from $i$ to $j$ is the reversal of the path from $j$ to $i$. Therefore, $\delta_{\lambda}= \delta_{-\lambda}$ since all $\be_{i,j}$ are eigenvectors of $\ad_{\hatF}$ and all $\be_{i,j}\in \gln$.  Since the undirected meander is connected,  it follows that  there  are no gaps in $\Ooms(\gln;\L)$ and thus $\Ooms(\gln;\L)=\{-s, -s+1, \ldots, s-1,s\}$, where $s$ is the largest eigenvalue in $\Ooms(\gln;\L)$. In this case the $\gln$ dimension vector 
\[ \dimvec(\gln;\L)=(\delta_{-s}, \delta_{-s+1}, \ldots , \delta_{s-1}, \delta_s)\]
is symmetric around zero. See Table \ref{tbl:bigger_examples} for an example illustrating these dimension vectors.
%%%
\begin{table}[!hbt]
	\footnotesize
	% % {a, b} = {{50}, {13, 37}};
	\begin{gather*}
		\begin{array}{@{\,}r|r@{\ }r@{\ }r@{\ \ }r@{\ \ }r@{\ \ }r@{\ \ }r@{\ \ }r@{\ \ }r@{\ \ }r@{\ \ }r@{\ \ }r@{\ \ }r@{\ \ }r@{\ \ }r@{\ \ }r@{\ \ }r@{\ \ }r@{\ \ }r@{\ \ }r@{\ }r@{\,}}
			\lambda & -10 &  -9 & -8 & -7 & -6 & -5 & -4 & -3 & -2 & -1 & 0 & 1 & 2 & 3 & 4 & 5 & 6 & 7 & 8 & 9 & 10 \\
			\hline
			\rule[1.5ex]{0pt}{1ex}
			\L(50 {\,|\,} 13,37) & & 1 & 5 & 15 & 34 & 63 & 101 & 144 & 186 & 220 & 240 & 240 & 220 & 186 & 144 & 101 & 63 & 34 & 15 & 5 & 1 
			\\[.25ex]
			\gl_{50};\L &  1 & 5 & 15 & 34 & 64 & 105 & 153 & 203 & 247 & 278 & 290 & 278 & 247 & 203 & 153 & 105 & 64 & 34 & 15 & 5 & 1
		\end{array}
		%
		% \\[4ex]
		%
		% {a, b} = {{16,34}, {13,17,20}};
		% \begin{array}{@{\,}r|r@{\ }r@{\ \ }r@{\ \ }r@{\ \ }r@{\ \ }r@{\ \ }r@{\ \ }r@{\ \ }r@{\ \ }r@{\ \ }r@{\ \ }r@{\ \ }r@{\ \ }r@{\ \ }r@{\,}}
			% \lambda & -7 & -6 & -5 & -4 & -3 & -2 & -1 & 0 & 1 & 2 & 3 & 4 & 5 & 6 & 7 \\
			% \hline
			% \rule[1.5ex]{0pt}{1ex}
			% \dimvec(\L) & & &  1 & 7 & 29 & 87 & 181 & 262 & 262 & 181 & 87 & 29 & 7 & 1 
			% \\[.25ex]
			% \dimvec(\gl_{50};\L) & 2 & 11 & 35 & 83 & 168 & 293 & 420 & 476 & 420 & 293 & 168 & 83 & 35 & 11 & 2
			% \end{array}
	\end{gather*}
	\caption{Dimension vectors for $\L$ and $\gl_{50};\L$, starting from 
		$\L=\L(50 {\,|\,} 13,37)$.
		%$\L=\L(16,34 {\,|\,} 13,17,20)$.
	}
	\label{tbl:bigger_examples}
\end{table}

%%%%%%%%%%%%%%%%%%%%%%%%%%%%
\section{Maximal Parabolics and Panyushev Moves}
\label{sec:panyushev}
%%%%%%%%%%

\subsection{Maximal parabolic subalgebras}
\label{subsec:maximal-parabolics}
%%%%%%%%%%

We now specialize to maximal parabolic subalgebras of $\sln$. Fix 
$1\leq a\leq n-1$, set $b=n-a$, and write
\[
\pab=\L((n)\mid(a,b))
\ \ \qand \ \ 
\puab = \L((b,a)\mid(n)).
\]
Thus, $\pab$ is the maximal parabolic subalgebra obtained by omitting the 
negative simple root vector $\be_{a+1,a}$ while $\puab$ is the maximal parabolic subalgebra obtained by omitting the 
positive simple root vector $\be_{b,b+1}$. In fact, $\pab \cong \puab$, since matrix rotation by $180^{\circ}$ induces an automorphism of $\sln$.\footnote{Specifically, it is an inner automorphism $\sigma:\sln \rightarrow \sln$ sending matrix $X$ to $JXJ$, where $J$ is the full-reversal permutation matrix (ones on the antidiagonal and zeros elsewhere). So $\L((b,a)\mid(n)) \xleftrightarrow{\,\sigma\,} \L((n)\mid(a,b))$, as required.} 
Nevertheless, it will be useful to have a shorthand for both upper- and lower-triangular parabolics in what follows. 

The following criterion determines exactly when these maximal parabolic
subalgebras are Frobenius.

\begin{theorem}[Elashvili {\cite{elashvili1982frobenius}}]
\label{th:elashvili-maximal}
The maximal parabolic subalgebra $\pab$ is Frobenius if and only if
\[
\gcd(a,b)=1.
\]
Equivalently, $\pab$ is Frobenius if and only if $\gcd(n,a)=1$.
\end{theorem}

\subsection{Definition and graph-theoretic reformulation}
\label{sec:panyushev-moves}

From Theorem \ref{th:index-from-meander}, we know that the index of any biparabolic subalgebra $\L(\valpha\mid \vbeta)$ of $\sln$ (not necessarily Frobenius and not necessarily maximal parabolic) can be computed combinatorially from its meander graph. An alternative but very useful inductive method to compute the index was given by Panyushev \cite{panyushev2001inductive}. In the original presentation, this method is Lie-theoretic and based on a theorem of Ra{\"\i}s concerning the index of certain semi-direct products, see \cite{rais1978index}. In what follows, we describe a particular case of \cite[Prop.~4.1]{panyushev2001inductive}, which is all that is needed for our results.

%%%%%%%%%%
Panyushev's reduction process takes a maximal parabolic subalgebra $\fp$ inside $\sln$ to a maximal parabolic subalgebra $\fp'$ inside $\mathfrak{sl}_{n'}$ with $n'<n$. 
We state it in the case $\fp$ is upper-triangular:
\begin{align}
\notag
\L((n)\mid(a,b)) &\longmapsto 
\begin{cases} 
\L((b-a,a)\mid(b))\quad\text{if}\quad a<b\\
\L((b,a-b)\mid(a))\quad\text{if}\quad a>b
\end{cases},
\intertext{or equivalently,}
\label{eq:reduction}
\mathfrak{p}(a,b) &\longmapsto 
\begin{cases} 
\mathfrak{p}\underline{(a,b-a)}\quad\text{if}\quad a<b\\[2ex]
\mathfrak{p}\underline{(a-b,b)}\quad\text{if}\quad a>b
\end{cases}.
\end{align}
In all cases the index is preserved. When $\pab$ is  Frobenius ({\it i.e.}, index zero) this is immediate from Theorem~\ref{th:elashvili-maximal}  and the elementary equalities $\gcd(a,b)=\gcd(a,b-a)=\gcd(a-b,b)$. 
%\[
%\gcd(a,b)=\gcd(a,b-a)\quad\text{if }a<b,
%\qquad\text{and}\qquad
%\gcd(a,b)=\gcd(a-b,b)\quad\text{if }a>b.
%\]

As exhibited in \cite{coll2015meander}, Panyushev's reduction process can be visualized as retracting a ``tongue'' stretching out over one end of the meander graph. See Figure \ref{fig:panyushev-reduction} for the passage from $\L(11\mid 7,4)=\pab[7,4]$ to $\L(4,3\mid 7)=\puab[3,4]$. %
\begin{figure}[!htb]
\centering
    \figA
\caption{Panyushev reduction $\pab[7,4] \mapsto \pab[\underline{3,4}]$.}\label{fig:panyushev-reduction}
\end{figure}

\subsection{Panyushev moves and rank polynomials}
\label{sec:panyushev_rank}
%%%%%%%%%%
It is possible to run the Panyushev reduction process in reverse---{extending} a tongue out over one end of the meander graph. In this direction, there is no need to check, for $\pab$, whether $a>b$ or $a<b$. Both extensions are possible. For convenience, after the reverse reduction, we rotate the shape $180^\circ$ so that the result is in upper-triangular form. We call both extensions \demph{Panyushev moves}. These moves can be regarded as branching rules that generate an infinite binary tree with $\pab[1,1]$ as root and Frobenius $\pab$ as nodes; that is, 
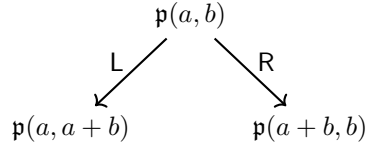
\begin{center}
\begin{tikzpicture}[level 1/.style={sibling distance=9em, level distance=4.25em}]
%[scale=0.90,>=latex']
    \node %root
    {$\displaystyle\pab$}
      % left child
      child {
        node {$\displaystyle\pab[a,a+b]$}
        edge from parent[thick,->,pos=.35]
          node[left,xshift=-3pt] {$\pL$}
      }
      child {
        node {$\displaystyle\pab[a+b,b]$}
        edge from parent[thick,->,pos=.35]
          node[right,xshift=3pt] {$\pR$}
      };
\end{tikzpicture}
\end{center}
where  $\pL$ and $\pR$ pull the left tongue or right tongue across the other, respectively.

This is reminiscent of the Calkin--Wilf tree \cite{calkin2000recounting,berstel1997sturmian}, which has the rational number $\frac{1}{1}$ as its root and is determined by the successor rules below.
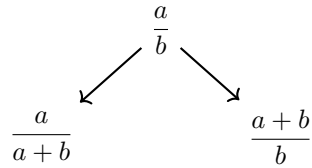
\begin{center}
\begin{tikzpicture}[level 1/.style={sibling distance=9em, level distance=4em}]
%[scale=0.90,>=latex']
    \node %root
    {$\dfrac{a}{b}$}
      % left child
      child {
        node {$\dfrac{a}{a+b}$}
        edge from parent[thick,->,pos=.4]
          node[left,xshift=-3pt] {}
      }
      child {
        node {$\dfrac{a+b}{b}$}
        edge from parent[thick,->,pos=.4]
          node[right,xshift=3pt] {}
      };
\end{tikzpicture}
\end{center}
As the Calkin--Wilf tree contains each positive rational number $a/b$ in reduced form exactly once, the Panyushev tree contains each Frobenius maximal parabolic $\pab$ exactly once.

In the context of our Panyushev tree, the reduction process described in \eqref{eq:reduction} simply identifies the unique parent of $\pab$. Iterating, we see that each Frobenius maximal parabolic $\pab$ is described by a unique path through the binary tree, up to the root $\pab[1,1]$. 
Taking the path just described in the opposite direction, from root down to $\pab$, defines a word $w$ on letters $\{\pL,\pR\}$ that instructs us how to reach $\pab$ from the root. We call $w$ the \emph{address} of $\pab$ and  define $\fp_w := \pab$.

The effect of Panyushev moves on meanders is easy to see but somewhat cumbersome to describe. Let $M_{a,b}$ $(= M_w)$ denote the ranked meander of the  Frobenius  $\pab$ with address $w$, and put $n=a+b$. The transformation $M_{a,b} \xmapsto{\pR} M_{n,b}$ proceeds as follows (cf.  Figure \ref{fig:panyushev-move-on-meanders}): First build an intermediary digraph $M'$ by 
 (i) duplicating the last $b$ nodes in $M_{a,b}$ and adding edges $(i,i')$, then  
 (ii) removing any edge $(i,j)$ with $i>j$ among the original last $b$ nodes and replacing with edge $(i',j')$. 
 Now relabel the  nodes of $M'$ so that node $i$ becomes $n+1-i$ and  node $i'$ becomes $i+b$. 
%%%
\begin{figure}[!hbt]
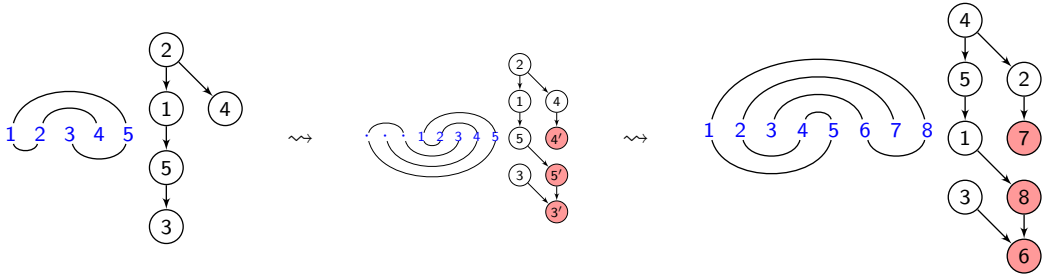

\centering
\figB
\caption{The Panyushev move $\pR$ acts on  $M_{2,3}=M_{\pR\pL}$ to produce $M_{5,3}=M_{\pR\pL\pR}$.}
\label{fig:panyushev-move-on-meanders}
\end{figure}

Note that $\rnk(i')=\rnk(i)-1$ for each new node $i'$ of $M'$. Since  the  rank differences ({\it i.e.}, path weights) between all pairs of original nodes are evidently the same in $M'$ as in $M$, the rank of each original node $i \in [1,n]$ is  one greater in $M'$ than in $M$. The end result is that the ranks of vertices $[n+1,n+b]$ in $M_{n,b}$ match those of  $[a+1,a+b]$ in $M_{a,b}$, while the ranks of vertices $[1,n]$ are one larger in $M_{n,b}$ than in $M_{a,b}$. 

For the  left move $M_{a,b} \xmapsto{\pL} M_{a,n}$ the intermediary $M'$ is constructed as above but with ``last $b$ nodes''  changed to  ``first $a$ nodes'' and  edges $(i,i')$ reversed to $(i',i)$.   The nodes  are then relabeled so that $i$ becomes $n+a+1-i$ and $i'$ becomes $i$.  Here we have $\rnk(i')=\rnk(i)+1$ for each $i' \in M'$,
and the rank of each original node $i \in [1,n]$ remains the same in $M'$ as in $M$.

After accounting for  relabeling, these observations yield the following.

\begin{proposition}
\label{th:meander-graph-growth-properties}
Let $M_{a,b}$ be the ranked meander for a Frobenius $\pab$ and let $n=a+b$. 
Then 
\(
\pL(M_{a,b})= M_{a,n}
\), and the ranks of vertices $[a+1,a+n]$ in $M_{a,n}$ match those of $[1,n]$ in $M_{a,b}$ while the ranks of $[1,a]$ are one greater in $M_{a,n}$ than in $M_{a,b}$.
Similarly,
\(
\pR(M_{a,b}) = M_{n,b},
\)
and the ranks of vertices $[n+1,n+b]$ in $M_{n,b}$ match those of  $[a+1,a+b]$ in $M_{a,b}$, while the ranks of vertices $[1,n]$ are one larger in $M_{n,b}$ than in $M_{a,b}$. 
\qed
\end{proposition}

A simple induction argument now reveals that nodes $[a+1,a+b]$ are clustered at the bottom of $M_{a,b}$. That is, while there may be large jumps in rank between any particular pair of these nodes, the set of achieved ranks is contiguous and contains 0. (See the left and right extremes of Figure \ref{fig:panyushev-move-on-meanders} for two illustrations of this phenomenon.)  Similarly, nodes $[1,a]$ are clustered at the top of $M_{a,b}$, {\it i.e.}, their ranks form a contiguous set containing the maximum.  Moreover, the minimum and maximum ranks are achieved exactly once each, for distinct nodes $i \in [a+1,a+b]$ and  $j  \in [1,a]$, respectively.

We are  interested in the distribution of ranks among the nodes of  $M_w=M_{a,b}$. As such, it is natural to consider the \emph{rank polynomial} 
$$
    \Rfull{w}(t) = \Rfull{a,b}(t) := \sum_{i \in [1,n]} t^{\rnk(i)}.
$$
The preceding discussion further suggests that we stratify this polynomial by introducing the upper and lower polynomials 
$$
\Rtop{w}(t) := \sum_{i \in [1,a]} t^{\rnk(i)}
\qquad\text{and}\qquad
\Rbot{w}(t) := \sum_{i \in [a+1,n]} t^{\rnk(i)},
$$
which evidently satisfy $\Rfull{w}(t)=\Rtop{w}(t)+\Rbot{w}(t)$.  For instance, from Figure~\ref{fig:panyushev-move-on-meanders} we see that $\Rfull{2,3}(t)=\Rfull{\pR\pL}(t)=t^3+2t^2+t+1$, with 
\begin{equation}
\label{eq:runningex1}
    \Rtop{\pR\pL}(t)=t^3+t^2
    \qquad\text{and}\qquad
    \Rbot{\pR\pL}(t)=t^2+t+1.
\end{equation}

The behaviour of these  polynomials under the Panyushev moves is dictated by Proposition~\ref{th:meander-graph-growth-properties}.  In particular, for any finite word $w$ in $\{\pL,\pR\}$, the proposition implies
\begin{gather}
\label{eq:meander-rankpolynomials}
    \left\{
    \begin{aligned}
        \Rtop{w\pL}(t) &= t\, \Rtop{w}(t) \\[.5ex]
        \Rbot{w\pL}(t) &= \Rfull{w}(t)=\Rtop{w}(t)+\Rbot{w}(t)
    \end{aligned}
    \right.
\quad\qand\quad
    \left\{
    \begin{aligned}
        \Rtop{w\pR}(t) &= t\,\Rfull{w}(t)=t\big(\Rtop{w}(t)+\Rbot{w}(t)\big) \\[.5ex]
        \Rbot{w\pR}(t) &= \Rbot{w}(t).
    \end{aligned}
    \right.
\end{gather}
Starting from $\Rtop{\varepsilon}(t)=t$ and $\Rbot{\varepsilon}(t)=1$, these rules recursively produce $\Rtop{w}(t)$ and $\Rbot{w}(t)$, and hence their sum $\Rfull{w}(t)$, for all $w$. 
Mimicking the Calkin--Wilf tree, we formally organize these pairs of polynomials as the \demph{rank rational functions} $\overline{R}_w(t):=\Rtop{w}(t)/\Rbot{w}(t)$ for each $w$.
 
%%%%%%%%%%%%%%%%%%%%%%%%%%%%
\section{Order Ideals of Fence Posets}
\label{sec:fence posets}
%%%%%%%%%%
Here we consider posets related to certain walks in the Cartesian plane and their connection to Frobenius Lie algebras. 

Let $\pU$ denote the northeast step $(+1,+1)$ and $\pD$ denote the southeast step $(+1,-1)$. Any finite word $w$ in $\pU$ and $\pD$ describes a jagged-line path starting at the origin and containing $n:=|w|+1$ nodes. This path  can be viewed as the Hasse diagram of a partially ordered set on $\{1,2,\ldots,n\}$, with nodes naturally  labeled  from left to right. This is known as the \demph{fence poset} of type $w$, denoted $P_w$. 

Recall that a subset $I$ of $P_w$ is an \demph{order ideal} if it is downward closed, that is $y\in I$ and $x<y$ imply $x\in I$. The set $\cJ_w$ of all such order ideals under containment order ($I \leq J \iff I \subseteq J$) is a ranked lattice    with rank function $\rnk(I):=|I|$. The \demph{rank polynomial} of $\cJ_w$ is then defined by
\[
	\Jfull{w}(t) := \sum_{I \in \cJ_w} t^{\rnk(I)}.
\]
For example,  taking $w$ to be the empty word $\varepsilon$ gives $P_\varepsilon=\{{\color{blue}1}\}$ (the singleton poset), $\cJ_\varepsilon=\{\emptyset,\{1\}\}$ and  $\Jfull{\varepsilon}(t) = t + 1$.  Figure \ref{fig:fence-poset-and-order-ideal} displays a more substantive example. 
Note that the leading and constant coefficients of $\Jfull{w}(t)$ are always $1$, corresponding to ideals $[1,n]$ and $\emptyset$, respectively. 
\begin{figure}[!hbt]
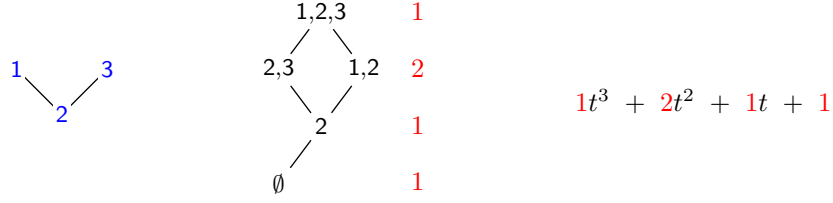

\centering
\figC
\caption{The fence poset $P_{\pD\pU}$, its lattice of order ideals $\cJ_{\pD\pU}$, and the rank polynomial $\Jfull{\pD\pU}(t)$.}
\label{fig:fence-poset-and-order-ideal}
\end{figure}

\subsection{An equivalence of trees}
\label{sec:tree-isomorphism}
%%%%%%%%%%
As for the Frobenius maximal parabolics, all fence posets may also be organized as an infinite binary tree. The root corresponds to  $P_\varepsilon=\{\mathsf{\color{blue}1}\}$. The successor rules append one extra step ($\pU$ or $\pD$) to the defining path.
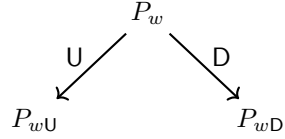
\begin{center}
\begin{tikzpicture}[level 1/.style={sibling distance=8.5em, level distance=4em}]
%[scale=0.90,>=latex']
    \node %root
    {$\displaystyle P_w$}
      % left child
      child {
        node {$\displaystyle P_{w\pU}$}
        edge from parent[thick,->,pos=.35]
          node[left,xshift=-3pt] {$\pU$}
      }
      child {
        node {$\displaystyle P_{w\pD}$}
        edge from parent[thick,->,pos=.35]
          node[right,xshift=3pt] {$\pD$}
      };
\end{tikzpicture}
\end{center}
Clearly, every fence poset $P_w$ appears somewhere in this tree, and its defining word $w$ may be viewed as an {address} providing instructions for how to reach $P_w$ from the root. This establishes a bijection between fence posets $P_w$ and Frobenius maximal parabolics $\fp_{w'}$, where $w'$ is the image of $w$ under the rewriting rule $(\pU,\pD) \leftrightarrow (\pL,\pR)$. In the remainder of this section, we show this connection is more than superficial. (We also suppress further mention of the rewriting rule.)

Let $w$ be any finite word on $\{\pU,\pD\}$ and set $n=|w|+1$. The element $n \in P_w$ naturally partitions $\cJ_w$ into  two classes  $\Top{\cJ_w} := \left\{I \in \cJ_w \mid {n}\in I\right\}$ and $\Bot{\cJ_w} := \left\{I \in \cJ_w \mid {n}\not\in I\right\}$. Note that  $\Bot{\cJ_w}$ is an order ideal of $\cJ_w$, so its elements are clustered at the bottom ranks. Likewise, $\Top{\cJ_w}$ is an ``upper order ideal'' and its elements are clustered at the top. 
Define the upper and lower rank polynomials of $\cJ_w$ by
$$
\Jtop{w}(t) := \sum_{I \in \Top{\cJ_w}} t^{\rnk{I}}
\qquad\text{and}\qquad
\Jbot{w}(t) := \sum_{I \in \Bot{\cJ_w}} t^{\rnk{I}}.
$$
For example, in
Figure \ref{fig:fence-poset-and-order-ideal} we have $\Top{\cJ_{\pD\pU}} = \{\{1,2,3\},\{2,3\}\}$ and $\Bot{\cJ_{\pD\pU}} = \{\{1,2\}, \{2\},\emptyset\}$, so
\begin{equation}
\label{eq:runningex2}
\Jtop{\pD\pU}(t)=t^3+t^2
\qquad\text{and}\qquad
\Jbot{\pD\pU}(t)=t^2+t+1. 
\end{equation}

The next proposition describes the behaviour of the partition $\cJ = \Top{\cJ} \sqcup \Bot{\cJ}$ under the successor rules $\pU$ and $\pD$.
All claims are immediate from our definitions. See Figure~\ref{fig:fence-poset-move} for an illustration.

\begin{proposition}
\label{th:order-ideals}
In the running notation we have
\begin{gather*}
    \left\{
    \begin{aligned}
        \Top{\cJ_{w\pU}} &= \left\{\{n+1\}\cup I \mid I \in \Top{\cJ_w}\right\} \\[.5ex]
        \Bot{\cJ_{w\pU}} &= \cJ_w
    \end{aligned}
    \right.
\quad\qand\quad
    \left\{
    \begin{aligned}
        \Top{\cJ_{w\pD}} &= \left\{\{n+1\}\cup I \mid I \in \cJ_w\right\} \\[.5ex]
        \Bot{\cJ_{w\pD}} &= \Bot{\cJ_w}
    \end{aligned}
    \right.
\end{gather*}
for all $w \in \{\pU,\pD\}^*$.
\qed
\end{proposition}
%%%
\begin{figure}[!hbt]
\centering
\figD

\caption{The transition from $\cJ_{w}$ to $\cJ_{w\pD}$ for $w=\pD\pU$.  The elements of $\Bot{\cJ_{w}}$ and $\Bot{\cJ_{w\pD}}$ are highlighted. Every element $I \in \cJ_w$ is augmented to $I' = I \cup \{4\} \in \Jtop{w\pD}$, with $\rnk(I')=\rnk(I)+1$.}
\label{fig:fence-poset-move}
\end{figure}

Recast in terms of rank polynomials, the proposition says we have the following successor rules:
\begin{gather}
\label{eq:fence-rankpolynomials}
    \left\{
    \begin{aligned}
        \Jtop{w\pU}(t) &= t\, \Jtop{w}(t) \\[.5ex]
        \Jbot{w\pU}(t) &= \Jtop{w}(t)+\Jbot{w}(t)
    \end{aligned}
    \right.
\quad\qand\quad
    \left\{
    \begin{aligned}
        \Jtop{w\pD}(t) &= t\big(\Jtop{w}(t)+\Jbot{w}(t)\big) \\[.5ex]
        \Jbot{w\pD}(t) &= \Jbot{w}(t),
    \end{aligned}
    \right.
\end{gather}
with $\Jtop{\varepsilon}(t)=t$ and $\Jbot{\varepsilon}(t)=1$. (Compare with \eqref{eq:meander-rankpolynomials}.)

\begin{theorem}
	\label{th:t-Calkin-Wilf-tree}
	For all $w\in\{\pL,\pR\}^*$, the meander rank polynomial $\Rfull{w}(t)$ is equal to the order ideal rank polynomial $\Jfull{w}(t)$.
	Moreover, the upper and lower polynomials $\Rtop{w}(t)$ and $\Rbot{w}(t)$ have greatest common divisor equal to 1.
\end{theorem}

\begin{proof}
The proof proceeds by induction on $|w|$ with base case at the root $\varepsilon$.  The triples $(\Rtop{w},\Rbot{w},\Rfull{w})$ and $(\Jtop{w},\Jbot{w},\Jfull{w})$  agree at the root $(t,1,t+1)$ and have the same successor rules by \eqref{eq:meander-rankpolynomials} and \eqref{eq:fence-rankpolynomials}. It follows they are equal for all $w$. These same successor rules also inductively establish the greatest common divisor claim.
\end{proof}

\subsection{$t$-Rational numbers and unimodality}
\label{sec:rank-unimodality}
%%%%%%%%%%
The final statement in Proposition~\ref{th:t-Calkin-Wilf-tree} is worth highlighting. It says our rational function $p_w/q_w$ built from the ranked meander $M_w=M_{a,b}$ is a $t$-analog of $a/b$. We provide a few more examples before we continue. 

Figure \ref{fig:panyushev_tree} shows a portion of our Calkin--Wilf tree, drawn as ranked meanders. 
\begin{figure}[!hbt]
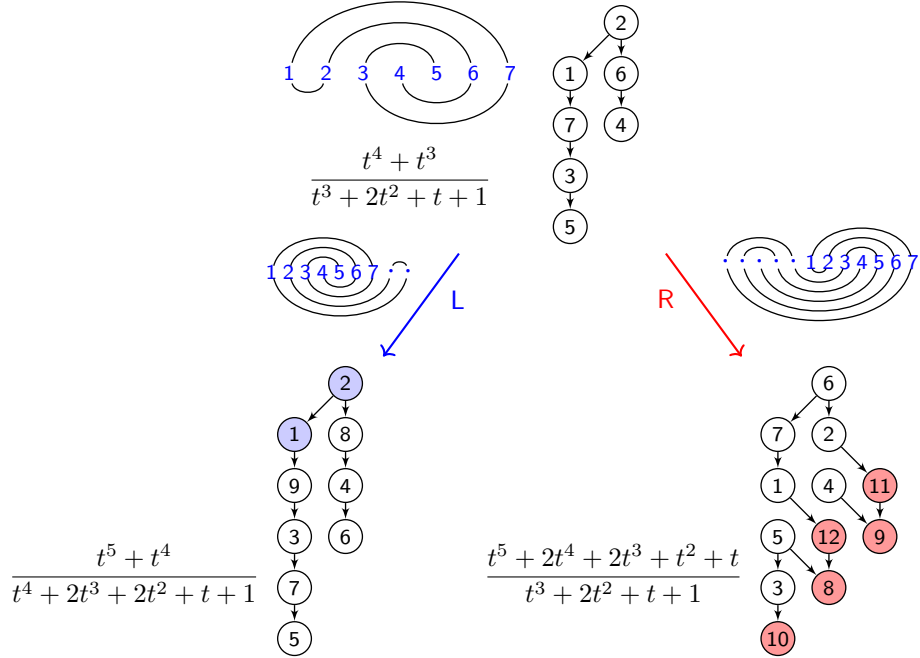

\centering
\figE
\caption{The effect of left and right Panyushev moves on ranked meanders $M_{w}$ and the corresponding $t$-rational numbers.} 
\label{fig:panyushev_tree}
\end{figure}
Figure \ref{fig:fenceposet_tree} shows a portion of the tree drawn as order ideal posets. 
(The top nodes shown in Figures \ref{fig:panyushev_tree} and \ref{fig:fenceposet_tree} have the common parent $\frac{{t}^3 + {t}^2}{{t}^2 + {t} + 1}$, appearing in Figures~\ref{fig:panyushev-move-on-meanders} and \ref{fig:fence-poset-and-order-ideal}.)
\begin{figure}[!hbt]
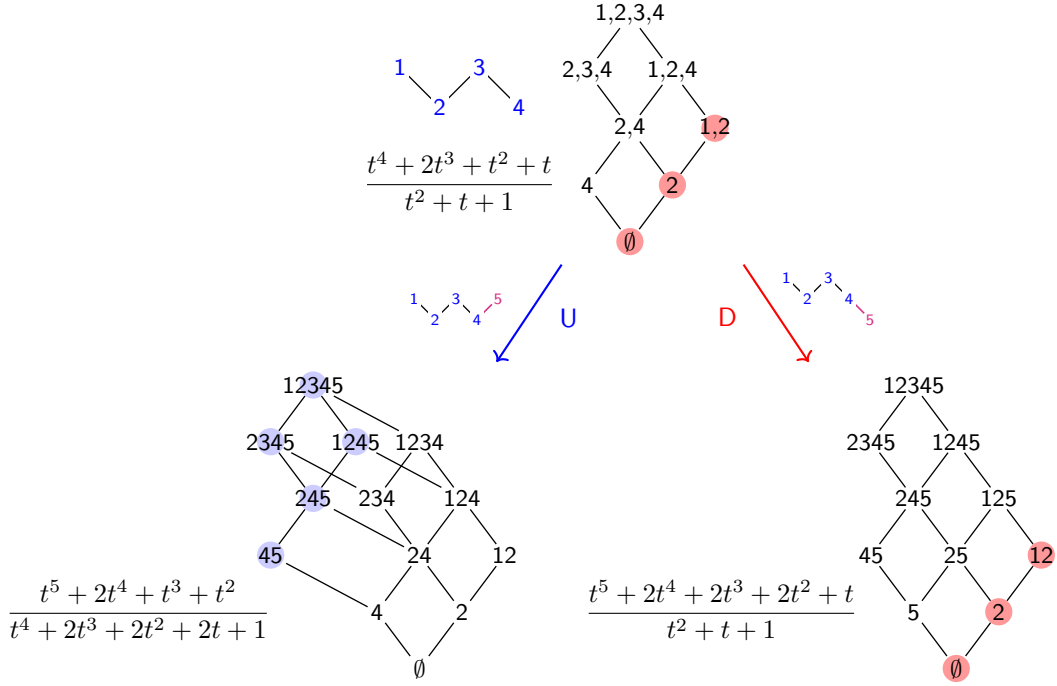

\centering
\figF
\caption{The effect of up and down moves on the lattice of order ideals and the corresponding $t$-rational numbers.}
\label{fig:fenceposet_tree}
\end{figure}
As the examples show, the coefficients appearing in the numerator and denominator are somewhat mysterious, but the successor rules do provide some rudimentary information about these polynomials. Some notation first: 
if $h(t) = h_mt^m + \cdots + h_1t + h_0$ is a polynomial in $\ZZ[t]$ with $h_m\neq0$, we write $\mathsf{supp}(h)$ for the subset of indices $i\in[0,m]$ with $h_i \neq0$.

\begin{corollary}\label{th:leading-terms-and-support}
Let $w=w_1\cdots w_l \in \{\pL,\pR\}^*$ and let $i$ and $j$, respectively, be the indices of the rightmost occurrences of $\pL$ and $\pR$ within $w$, taking either  to be 0 if the corresponding symbol is not present. Writing $p_w/q_w$ for the rational function $\overline{R}_{w}(t)$, we have
\begin{gather*}
    \mathsf{supp}(p_{w}) = [l-j+1,l+1]
    \ \ \qand \ \ 
    \mathsf{supp}(q_{w}) = [0,i] 
\end{gather*}
Furthermore, the leading and trailing coefficients on these supports equal 1 in all cases. \qed
\end{corollary}

%%%%% (keep hidden but don't delete) %%%%%
% Our $t$-Calkin--Wilf tree is different from the tree introduced by Bates and Mansour \cite{bates2011q-calkinwilf}, which starts from a $t$-deformation of Calkin's formula for enumerating the tree by rows. 
%%%%%%%%%
We note that our rational function analogs are closely related to those of  
Morier-Genoud and Ovsienko \cite{morier2020q-deformed}, who define a rational function $[\frac{a}{b}]_q$ of any $a/b>1$ from a $q$-analog of its regular continued-fraction representation. Compare the two versions of $25/11$, appearing below.
\begin{align}
\label{eq:t-rationals}
    \bigl[\tfrac{25}{11}\bigr]_t &= \dfrac{t^8+2t^{7\,}+4t^{6\,}+5t^{5\,}+5t^{4\,}+4t^{3\,}+3t^2+{t}}{t^{5}+2t^{4}+3t^{3\,}+2t^{2\,}+2{t}+1}
    \\[1ex]
\label{eq:q-rationals}
    \bigl[\tfrac{25}{11}\bigr]_q &= \dfrac{q^7+2q^6+4q^5+5q^4+5q^3+4q^2+3q+1}{q^5+2q^4+3q^3+2q^2+2q+1}
\end{align}
It happens that the continued-fraction representation of a rational number is related to its address in the Stern--Brocott tree. See \cite{deluca2011christoffel} for details, where De Luca and Reutenauer also give a precise dictionary between addresses of a rational number in the Stern--Brocott tree and in the Calkin--Wilf tree.
See \cite{aval2025q-analogs} for the $q$-analog of the dictionary.\footnote{We thank S\'ebastien Labb\'e for sharing this reference; in it, the authors partition $\cJ$ in a manner analogous to our $\cJ^+\sqcup \cJ^+$ to reach some of their conclusions.}

The reader will note that, in all examples appearing thus far, the coefficient sequences for all numerators and denominators are unimodal. Likewise for the coefficients of the full rank polynomial $R_w(t)$. 
This phenomenon appears as a conjecture in \cite{morier2020q-deformed}; it was settled by Kantarc{\i}~O\u{g}uz and Ravichandran.

\begin{theorem}[\cite{kantarci2023rank}, Theorem 1.2]
\label{th:rank-unimodal}
The rank polynomial $\Jfull{w}(t)$ corresponding to any fence poset $P_w$ is unimodal.
\end{theorem}

We give a few further remarks about these $t$-rational numbers, for the interested reader. These will not be needed in what follows.

\begin{remark}
An even stronger result than Theorem \ref{th:rank-unimodal} appears in \cite{kantarci2023rank}, namely that the rank sequences are interlacing as first conjectured in \cite{mcconville2021rank-unimodality}. 
From Proposition \ref{th:t-Calkin-Wilf-tree}, we know that every $\Rtop{w}(t)/\Rbot{w}(t)$ appearing in the tree is distinct. Might that also hold for $\Rfull{w}(t)$? No, as $R_{3,13}(t) = R_{5,11}(t) = t^7 +  2t^6 + 3t^5 +3t^4 + 3t^3 +2t^2+ t + 1$.
As already noted in \cite{morier2020q-deformed}, the numbers ${[\frac{a}{b}]}_t$ are genuine rational number deformations, not merely ratios $[a]_t / [b]_t$ of integer deformations. For example, the denominator of \eqref{eq:t-rationals} disagrees with that of ${[\frac{5}{11}]}_t$, which is $t^6+2t^5+2t^4+2t^3+2t^2+t+1$. 
See \cite{mcconville2026hyperbinary} for additional work on the $t$-rational numbers.
\end{remark}

%%%%%%%%%%%%%%%%%%%%%%%%%%%%
\section{Good Gradings \& Proof of Main Result}
\label{sec:good gradings}

In this section we prove Theorem \ref{mod-conj}. The proof uses the unimodality of the rank polynomial $\Rfull{a,b}(t)$ of the meander together with the notion of a good grading of a Lie algebra.

\subsection{Even good gradings and pyramids}
Let $\mathfrak g$ be a finite-dimensional Lie algebra over a field of characteristic zero. Following Elashvili and Kac \cite{elashvili2005classification}, a $\ZZ$-grading $\mathfrak g =\bigoplus _{k\in \ZZ}\mathfrak g_k$ is called \demph{good} if there is an element $e\in \mathfrak g_2$ such that the following two properties hold
\begin{itemize}
\item $\ad_e: \mathfrak g_j\rightarrow \mathfrak g_{j+2}$ is injective for $j\leq -1$,
\item $\ad_e: \mathfrak g_j\rightarrow \mathfrak g_{j+2}$ is surjective for $j\geq -1$.
\end{itemize}
In the case that these properties hold, the element $e$ is called good.  

A good grading is called \demph{even} if $\mathfrak g_k=0$ for all odd $k$.  Only even good gradings are needed here. In this case, the only non-zero graded components of $\mathfrak g$ have even degree. Since the nilpotent element $e$ has degree two, we may divide all degrees by $2$ and assume that $e\in \mathfrak g_1$. We thus can assume $e\in \mathfrak g_1$ and that
\begin{itemize}
\item $\ad_e: \mathfrak g_j\rightarrow \mathfrak g_{j+1}$ is injective for $j\leq -1$,
\item $\ad_e: \mathfrak g_j\rightarrow \mathfrak g_{j+1}$ is surjective for $j\geq 0$.
\end{itemize}
We will use this rescaled convention in what follows. Two consequences of a rescaled even good grading are 
\begin{gather}\label{eq:good-is-unimodal}
\begin{aligned}\dim\mathfrak g_j & \leq\dim\mathfrak g_{j+1} \quad\text{for }j\leq -1,\quad \text{and}\\
\dim\mathfrak g_j & \geq\dim\mathfrak g_{j+1}
\quad\text{for }j\geq 0.
\end{aligned}
\end{gather}
It follows that the sequence of dimensions of the graded components of $\mathfrak g$ is unimodal around $0$.

We now restrict to the case of even good gradings of $\sln$, or, equivalently, of $\gln$. Elashvili and Kac \cite[Proposition~4.3]{elashvili2005classification} show that a unimodal sequence of size $n$ gives rise to a good grading of $\gln$. Suppose $\bc=(c_0,\ldots ,c_m)$ is a unimodal sequence of integers such that
\[0<c_0\leq c_1\leq \cdots \leq  c_k\geq  c_{k+1}\geq \cdots \geq c_m>0\qquad \text{and}\qquad c_0+\cdots +c_m=n.\]
The sequence determines a {pyramid} in the sense of \cite[Section~4]{elashvili2005classification} which can be visualized as follows. A \demph{pyramid} consists of $n$ equal-sized boxes arranged in columns numbered $0, ..., m$ from left to right, with the $j$-th column having $c_j$ boxes. The {pyramid} perspective is to treat the boxes as being arranged in rows (with the length of row $i$ being the number of elements in the sequence $\bc$ that are equal to or greater than $i$). 
Say $P$ is a \demph{filled pyramid} if $P$ comes from a unimodal sequence, as just described, and its boxes are filled with the integers $[1,n]$ in some way. See \eqref{eq:filled-pyramid} for a filled pyramid with row lengths $(5,3,2)$ associated to the unimodal sequence $(2,3,3,1,1)$.
\begin{gather}
\label{eq:filled-pyramid}
\raisebox{-.5\height}{\fskyline{{7,3},{2,6,10},{4,1,9},{8},{5}}}
\end{gather}
The good grading of $\gln$ associated to a filled pyramid $P$ is by \demph{column weight}: the matrix unit $\be_{i,j}$ has column weight $\col(j)-\col(i)$, where $\col(i)$ and $\col(j)$ denote the column indices of $i$ and $j$ within $P$, respectively. 
The corresponding good element $e\in(\gln)_1$ is given by 
\begin{equation}
\label{eq:nilpotent-e}
    e=\sum_{i,j}\be_{i,j},
\end{equation}
where the sum is over all pairs $i,j$ such that $i$ and $j$ are consecutive entries in any row, reading from left to right. 

\begin{example} 
Consider the filled pyramid $P$ from \eqref{eq:filled-pyramid}, associated to the unimodal sequence $(2,3,3,1,1)$ for $\fg=\gl_{10}$. The good element $e$ for $\fg$ is 
\[
    e =\be_{7,2}+\be_{2,4}+\be_{4,8}+\be_{8,5}
    +\be_{3,6}+\be_{6,1}
    +\be_{10,9}.
\]
The graded dimension sequence for $\fg$ associated to $P$ is
\[
    \begin{array}{c|rrrrrrrrr}
    \lambda & -4 & -3 & -2 & -1 & 0 & 1 & 2 & 3 & 4 \\
    \hline
    d_\lambda & 2 & 5 & 12 & 19 & 24 & 19 & 12 & 5 & 2
    \end{array}.
\]
%where, {\it e.g.}, the number $d_1=19$ is the sum of products of consecutive column heights, $2\cdot3 + 3\cdot3 + 3\cdot1 + 1\cdot 1$.
The elements 
$\bigl\{\be_{5,3}, \be_{5,7} \bigr\}$ span $\fg_{-4}$, 
and an elementary computation shows that $\ad_{e}(\fg_{-4})$ is spanned by
$
    \bigl\{\be_{8,3}-\be_{5,6}, \ 
    \be_{8,7}-\be_{5,2} 
    \bigr\}.
$
Hence $\ad_e: \fg_{-4}\rightarrow \fg_{-3}$ is injective.
\end{example}

\begin{remark}\label{rem:identical-or-conjugate} 
By construction, two fillings of a pyramid with identical column contents yield the same good grading of $\gln$. More generally, different fillings of the same pyramid may have different good elements $e$ but all have the same Jordan type and determine conjugate gradings of $\gln$. Specifically, the Jordan form of any good $e$ is associated to the partition $(\rho_1, \ldots, \rho_k)$ where $\rho_j$ is the number of entries in row $j$ of the pyramid, with $\rho_1$ representing the bottom row. Thus each row of the pyramid gives a Jordan chain for $e$. 
\end{remark}

\subsection{The meander grading is good}
Now consider once again a Frobenius maximal parabolic subalgebra $\pab$ of $\sln$. There are two natural gradings of $\gln$ associated to $\pab$. The first is determined by the eigenvalues of $\ad_{\hatF}$ which are computed as path weights in the ranked meander $M_{a,b}$. The second is determined by a pyramid associated to $M_{a,b}$, with grading computed as column weights for certain fillings. To match conventions in \cite{elashvili2005classification}, we enumerate the ranks in the graph from highest to lowest. That is, if $\Rfull{a,b}(t)=r_mt^m+\cdots + r_1t + r_0$ is the rank polynomial, we put $\bc=(c_0,\ldots, c_m) = (r_m, \ldots, r_0)$. By Theorem \ref{th:rank-unimodal} $\bc$ is unimodal and $c_0+\ldots +c_m=n$, so it determines a pyramid. The next result shows how these gradings of $\gln$ are related.

\begin{proposition}
\label{th:meander-grading-good}
For a Frobenius maximal parabolic $\pab$, the grading of $\gln$
determined by $\ad_{\widehat F}$ is a good grading.
\end{proposition}
\begin{proof}
 Fill the columns of the pyramid associated to $\pab$ with the integers $[1,n]$ in any way that respects their ranks in $M_{a,b}$. Specifically, if $\rnk(i)=i'$ in $M_{a,b}$, then place $i$ in column $m-i'$ of the pyramid. Then $\be_{i,j}$ has column weight $(m-j')-(m-i') = i'-j'$ for the corresponding good grading of $\gln$. Now recall that the integer grading on $\gln$ determined by $\ad_{\hatF}$ can be computed as difference in ranks within $M_{a,b}$, {\it i.e.}, $\ad_{\hatF}\be_{i,j}=\lambda_{i,j}\be_{i,j}$, when $\lambda_{i,j} = \rnk(i)-\rnk(j) = i'-j'$. Therefore the good grading determined by this filled pyramid is exactly the same as the grading determined by $\ad_{\hatF}$.
\end{proof}
\begin{corollary}
\label{cor:conjecture-part2-maximal}
If $\pab$ is Frobenius, then $\dimvec(\gln;\pab)$
is unimodal and symmetric about $0$.
\end{corollary}

\begin{proof}
The symmetry about $0$ was established in Section~\ref{subsec:principal-spectrum}.
By Proposition~\ref{th:meander-grading-good}, the grading of $\gln$
determined by $\ad_{\widehat F}$ is a good grading.
Hence the dimensions of its eigenspaces are nondecreasing up to degree $0$. By symmetry, it follows that the sequence is unimodal. 
\end{proof}
Next we turn to part~{\rm (1)} of Theorem \ref{mod-conj}. First note that Proposition~\ref{th:meander-grading-good} gives a good grading of $\gln$ associated to any Frobenius maximal parabolic $\pab$. However, some care needs to be taken to show that this grading can be used to show that $\dimvec(\pab)$ is unimodal around $\frac12$. The subtlety is that while a good nilpotent $e$ acts injectively on negative weight spaces of the $\gln$ grading, it does not immediately follow that it also acts injectively on negative weight spaces of $\pab$. 
In order to establish this we will show $e$ can be chosen to be an element of $\pab$.

Fix a Frobenius maximal parabolic $\pab$ and its associated pyramid. Recall that a good nilpotent $e$ for the grading on $\gln$ determined by $\ad_{\hatF}$ is determined from the filled rows of the pyramid, see  \eqref{eq:nilpotent-e}. The matrix units not contained in $\pab$ are those $\be_{i,j}$ with $i\in \{a+1,\ldots, n\}$ and $j\in \{1, \ldots, a\}$. So to ensure that $e\in \pab$, we must find a filled pyramid $P$ satisfying the following.
\begin{gather}
\label{eq:compatibility-criteria}
\begin{tabular}{p{.89\textwidth}}
\emph{(Compatibility criterion.)} $P$ is a filled pyramid for Frobenius maximal parabolic $\pab$, and no entry from $\{a+1,\ldots, n\}$ immediately precedes one from $\{1, \ldots, a\}$ in any row of~$P$.
\end{tabular}
\end{gather}

To this end define the subsets
\[
A=\{1, \ldots, a\}, \qquad B=\{a+1, \ldots, n\},\qquad \text{and} \qquad C=\{1, \ldots, n\},
\]
which are naturally partitioned by the filling. Concretely, put $A=\bigsqcup_{t=0}^m A_t$, with $i\in A_t$ if and only if $\col(i)=t$; and likewise for $B$ and $C$. Note that $(|C_0|, \ldots, |C_m|)$ is precisely the unimodal sequence $\bc$ for the pyramid. Define two additional sequences $\ba$ and $\bb$ analogously. 
% Note from Corollary \ref{th:leading-terms-and-support} that beginning entries in $\bb$ and ending entries in $\ba$ will be zero.
It follows from Proposition \ref{th:t-Calkin-Wilf-tree} and Theorem \ref{th:rank-unimodal} that both are unimodal and we have $a_t + b_t = c_t$ for all $t$.
We say that $[\ba, \bb, \bc]$ is a \demph{good triple} if the following two properties hold
\[
    c_s\leq c_{t}\Longrightarrow b_s\leq b_{t} \qquad\text{and}\qquad c_s\geq c_{t} \Longrightarrow a_s\geq a_{t} \qquad \text{for all}\qquad 0\leq s<t\leq m.
\]
See Figure \ref{fig:good-triples}. 
Equivalently, 
\begin{gather}
    \label{eq:good-triple-with-max}
    b_s \leq b_t + \max\{0,c_s-c_t\} \qquad \text{for all}\qquad 0\leq s<t\leq m.
\end{gather}
%%%
\begin{figure}[!hbt]
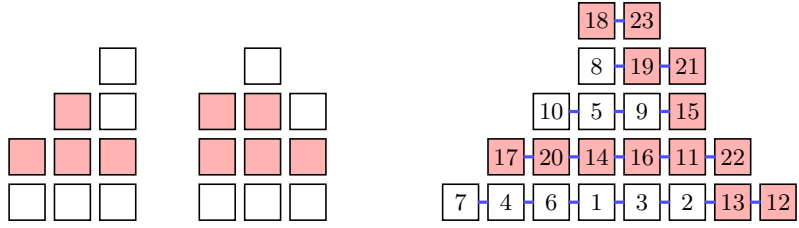

\centering
\figG

\caption{Three triples $[\ba, \bb, \bc]$, with $\ba$ and $\bb$ represented by unfilled and filled blocks, respectively. The first fails the good-triple condition between columns 1 and 2; the second fails between 0 and 2. 
The last triple is good. (It comes from the ranked meander $M_{10,13}$ and is filled using the algorithm in the proof of Proposition \ref{prop:good-e}.)}
\label{fig:good-triples}
\end{figure}

We next show that the triple $[\ba, \bb, \bc]$ associated to $\pab$ is always good. Without loss of generality, we may assume $a>b$. Indeed, the outer automorphism of $\sln$ induced by the reversal of the Dynkin diagram interchanges $\pab$ and $\mathfrak{p}(b,a)$ and it reverses the ranks in their associated meanders. The two defining conditions of a good triple are swapped and thus the triple associated to $\pab$ is good if and only if the same is true for $\mathfrak{p}(b,a)$.

Suppose $a>b$. Then from Panyushev reduction we have 
\[
\L((\bar{n})\mid(\bar{a}, \bar{b})) \xmapsto{\ \pR\ }\L((n)\mid(a, b)) \quad \text{with}\quad \bar{n}=a,\,\,\bar{a}=a-b,\,\,\bar{b} = b.\]
Let
$[\bar{\ba},\bar{\bb},\bar{\bc}]$
denote the triple associated to the smaller parabolic.
Proposition \ref{th:meander-graph-growth-properties} provides that 
\[
a_r=\bar{c}_r,\qquad b_r=\bar{b}_{r-1},\qquad \text{and}\qquad c_r=\bar c_r+\bar b_{r-1}. 
\]

\begin{proposition}
\label{prop:good-triple}
The triple $[\ba, \bb, \bc]$ associated to every
Frobenius maximal parabolic $\pab$ is good.
\end{proposition}

\begin{proof}
The proof is by induction on $n$. For $\mathfrak{p}(1,1)$ the result is clear. The predecessor of $\pab$ in the Panyushev tree is $\mathfrak{p}(\bar a, \bar b)$ with $\bar n =a$, $\bar a = a-b$ and $\bar b =b$. Since $\bar n<n$ the triple $[\bar{\ba},\bar{\bb},\bar{\bc}]$ of the predecessor algebra is good. In what follows we adopt the convention that any term with a negative subscript is zero. 

Assume that $c_s\leq c_{t}$. We need to show that $b_s\leq b_{t}$. Since $b_s=\bar{b}_{s-1}$ and $b_{t}=\bar{b}_{t-1}$, this is equivalent to showing that $\bar{b}_{s-1}\leq \bar{b}_{t-1}$. We argue by contradiction, supposing
\begin{equation}\label{contradiction}
    \bar{b}_{s-1}>\bar{b}_{t-1}.
\end{equation}
	
By induction, the triple $[\bar{\ba},\bar{\bb},\bar{\bc}]$ is good and so we have the implication
\[
    \bar{c}_{s-1}\leq \bar{c}_{t-1}\quad \Longrightarrow \quad \bar{b}_{s-1} \leq \bar{b}_{t-1}
\]
which, by contrapositive, is equivalent to
\[\bar{b}_{s-1}>\bar{b}_{t-1}\quad \Longrightarrow \quad\bar{c}_{s-1}>\bar{c}_{t-1}.\]
The unimodality of the $\bar{c}$-sequence implies that
\[\bar{c}_{s-1}>\bar{c}_{t-1} \quad \Longrightarrow \quad \bar{c}_s\geq \bar{c}_{t}.\]
Thus 
\begin{equation}\label{unimodal}\bar{b}_{s-1}>\bar{b}_{t-1}\quad \Longrightarrow \quad \bar{c}_s\geq \bar{c}_{t}.\end{equation}
Now we can reassemble the pieces and get a contradiction:
\[c_s = a_s+b_s= \bar{c}_s+\bar{b}_{s-1}\quad \text{and}\quad c_{t} = a_{t}+b_{t}= \bar{c}_{t}+\bar{b}_{t-1}.\]
Therefore,
\[c_{t}-c_{s}=\left(\bar{c}_{t}-\bar{c}_s\right) + \left(\bar{b}_{t-1} -\bar{b}_{s-1}\right).\]
Now the first summand of the last equation is non-positive by \eqref{unimodal}, and the second summand is negative by \eqref{contradiction}. Thus their sum is negative and this contradicts our assumption that $c_s\leq c_{t}.$ 
The second good condition that 
\[c_s\geq c_{t} \Rightarrow a_s\geq a_{t}\]
is proved in a similar way.
\end{proof}

Now that we know that any Frobenius $\pab$ has a good triple ${[\ba, \bb, \bc}]$, we can show that there is a filling of the associated pyramid with the property that $e$ lies in $\pab$.

\begin{proposition}
\label{prop:good-e}
Let $\pab$ be a Frobenius maximal parabolic subalgebra. Then there exists a filling of its associated pyramid such that the corresponding good element $e$ lies in $\pab$.
\end{proposition}

\newcommand{\sfA}{\ensuremath{\mathsf{A}}}
\newcommand{\sfB}{\ensuremath{\mathsf{B}}}

\begin{proof}
Let $P$ be the pyramid for $\pab$, let $[\ba,\bb,\bc]$ be a good triple, and recall the subsets $A_t,B_t,C_t$ $(0\leq t \leq m)$ from the running discussion. We give an algorithm that fills each cell of $P$ with one of the symbols $\{\sfA,\sfB\}$ in such a way that column $t$ contains $a_t$ $\sfA$'s and $b_t$ $\sfB$'s, and no $\sfA$ appears immediately to the right of a $\sfB$. The result  follows, since any labeling of the $\sfA$ cells with $A_t$ and the $\sfB$ cells with $B_t$ will then yield a good grading determined by $\ad_{\hatF}$ that  satisfies~\eqref{eq:compatibility-criteria}.

Some terminology will be helpful. Let us index  cells of $P$ by  pairs $(t,h)$, where $t \in [0,m]$ is a column index and $h \in [1,c_t]$ is a row index ({\it i.e.}, height).  If $P$ is partially filled with $\sfA$'s and $\sfB$'s, then we say a cell $(t,h) \in P$ is \emph{forced} if the cell $(t-1,h) \in P$ immediately to its left contains a $\sfB$.

The algorithm processes columns  from left to right.  After columns $0$ through $t-1$ have been filled, the next column $t$ is filled with $b_t$ $\sfB$'s and $a_t$ $\sfA$'s in stages as follows:  
\begin{itemize}
\item[(i)]    
The symbol $\sfB$ is placed in every forced cell. (There are always enough for this. See below.)
\item[(ii)] The remaining $\sfB$'s (if any) are placed in the topmost available empty cells. 
\item[(iii)] The remaining empty cells are filled with $\sfA$'s.
\end{itemize} 
Clearly this algorithm never places an $\sfA$ to the right of a $\sfB$, so we need only verify the claim made in Stage (i). Namely, one never encounters a column $t$ with more than $b_t$ forced cells.
Having filled columns $0,\ldots,t-1$, consider column $t$ ({\it e.g.}, column 5 in Figure~\ref{fig:building-a-filling}).
Let $H$ be the set of all row indices $h \in [1,c_{t}]$ such that cell $(t,h)$ is forced. 
We argue that $|H| \leq b_t$.  If $H=\emptyset$ then this is clear. Otherwise,
for  $h \in H$ let $s_h \leq t-1$ be the leftmost column such that cell $(s_h, h)$ contains $\sfB$.   Now let  $s := \max\{s_h \,:\, h \in H\}$ and suppose this maximum is attained at $h=h^*$. 
Note that cell $(s,h^*)$ contains $\sfB$ and is not forced. 

\begin{figure}[h!bt]
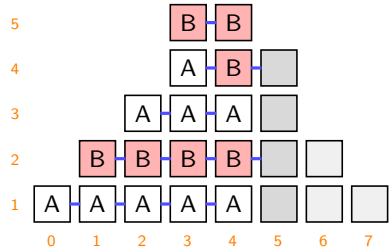

	\figH
    \caption{A candidate offending column in the proof of Prop.~\ref{prop:good-e}. Here $t=5$, $H=\{2,4\}$, $s=4$, $h^*=4$, $S_0 = \{(4,2),(4,4)\}$, and $S_1=\{(4,5)\}$. 
    The proof argues that $|H|\leq b_5$. See Fig.~\ref{fig:good-triples} for a proper filling.}
    \label{fig:building-a-filling}    
\end{figure}

Consider the disjoint sets of cells  $S_0:=\{(s,h)\,:\, h \in H\}$
and $S_1 := \{(s,k) \,:\, k \in (c_t,c_s]\}$ within column $s$.
Maximality of $s$ ensures each of the cells in $S_0$ contains $\sfB$.  We claim that each cell $(s,k) \in S_1$ also contains $\sfB$. 
Indeed, if $k \in (c_t,c_s]$ then  row $k$ of $P$ meets column $s$ but not column $t$. Since $P$ is a pyramid, and since every row $h \in H$ meets both columns $s$ and $t$, it follows that $k > h$ for all $h \in H$.  In particular, $k > h^*$.   Since  cell $(s,h^*)$ is not forced, the $\sfB$ therein must have been placed in Stage (ii) and thus cell $(s,k)$ already contained a $\sfB$ at the time. 

We have shown there are at least $|S_0|+|S_1|=|H|+\max\{0,c_s-c_t\}$ occurrences of symbol $\sfB$ in column $s$. The good triple condition~\eqref{eq:good-triple-with-max} therefore gives 
$$
b_t + \max\{0,c_s-c_t\} \geq b_s \geq |H|+\max\{0,c_s-c_t\},
$$
hence $|H| \leq b_t$, as desired.
\end{proof}

We can now use the previous results to prove Part~{\rm (1)} of Theorem~\ref{mod-conj}.
 
\begin{corollary}
\label{cor:conjecture-part1-maximal}
If $\pab$ is Frobenius, then $\dimvec(\pab)$
is symmetric and unimodal about $\frac12$.
\end{corollary}

\begin{proof}
Let
\[
\pab=\bigoplus_{j\in\mathbb Z}\mathfrak p_j(a,b)
\]
be the grading determined by $\ad_{\widehat F}$.
By Proposition~\ref{prop:good-e}, the corresponding good element $e$
may be chosen so that $e\in\pab$. Now $
\ad_e:
\mathfrak p_j(a,b)\longrightarrow\mathfrak p_{j+1}(a,b)
$
is injective for $j\leq -1$ as it is the restriction of the corresponding injective map on negative weight spaces of $\gln$. Thus
$\dim\mathfrak p_j(a,b)\leq
\dim\mathfrak p_{j+1}(a,b)$ for all $j\leq-1.$ Now the duality from the Kirillov form gives $ \dim\mathfrak p_j(a,b)= \dim\mathfrak p_{1-j}(a,b).$
Thus the dimensions of the positive weight spaces are weakly decreasing after degree $1$. The duality also gives $\dim\mathfrak p_0(a,b)=\dim\mathfrak p_1(a,b).$
Thus $\dimvec(\pab)$ is symmetric and unimodal about
$\frac12$.
\end{proof}

Combining Corollaries \ref{cor:conjecture-part2-maximal} and \ref{cor:conjecture-part1-maximal} establishes the validity of Theorem~\ref{mod-conj}.

%%%%%%%%%%%%%%%%%%%%%%%%%%%%
\section{Further Remarks}
\label{sec:further remarks}

We conclude by comparing and contrasting our type A maximal parabolic results with known phenomena in more general Frobenius biparabolic subalgebras. 

\emph{On good gradings.} Let $\L=\pab$ be Frobenius. We have shown that a generic element $e\in \L_1$ acts injectively on the spaces $\L_\lambda$ for $\lambda < 0$. Computational evidence suggets it also acts surjecively on $\L_\lambda$ for $\lambda > 0$. To give a good grading, we would also need $\ad_e:\L_0\to\L_1$ to be a bijection. This often happens, but not always---meaning a different notion of good grading is needed for non-semisimple Lie algebras. 

\emph{On unimodality.} A Frobenius biparabolic $\L(\valpha|\vbeta)$ was reported in \cite{coll_counterexample}   where $\dimvec(\L)$ is not unimodal. This finishes answering Duflo's question in type A. The rank polynomial for $M(\valpha|\vbeta)$ is not even unimodal in this case. Given the critical role it played for studying $\dimvec(\pab)$, it would be interesting to find another combinatorial entrypoint---analogous to fence posets and their order ideals---for the further study of $\dimvec{\L}$.

\emph{Beyond type A.}
Joseph developed a substantial general theory of biparabolic subalgebras of semisimple Lie algebras, including their index, semi-invariants, and adapted pairs  \cite{joseph2006semi, joseph2007semi, joseph2015integrality}.

\bibliography{biblio}
\bibliographystyle{abbrv}

\end{document}